\documentclass[
]{amsart}
\usepackage{algorithm, algorithmic}
\usepackage{amsmath}\usepackage{amsthm}\usepackage{amssymb}\usepackage{mathrsfs}\usepackage{amscd}\usepackage{graphicx}\usepackage{subfigure}\usepackage{amsfonts}\usepackage{amsxtra}\usepackage{color}
\usepackage{amscd}

\DeclareMathOperator{\rank}{rank}

\DeclareMathOperator{\trace}{trace}

\DeclareMathOperator*{\spann}{span}

\DeclareMathOperator{\Pol}{Pol}

\theoremstyle{definition}\newtheorem{definition}{Definition}[section]\newtheorem{remark}[definition]{Remark}
\theoremstyle{plain}\newtheorem{theorem}[definition]{Theorem}\newtheorem{lemma}[definition]{Lemma}\newtheorem{proposition}[definition]{Proposition}

\newcommand{\R}{\mathbb{R}}\newcommand{\C}{\mathbb{C}}

\newcommand{\Gkd}{{\mathcal G}_{k,d}}

\begin{document}

\title{Signal reconstruction from the magnitude of subspace components}

\author[C.~Bachoc]{Christine Bachoc}

\address[C.~Bachoc]{Univ.~Bordeaux, IMB, UMR 5251, F-33400 Talence, France}
\email{Christine.Bachoc@math.u-bordeaux1.fr}

\author[M.~Ehler]{Martin Ehler}
\address[M.~Ehler]{University of Vienna,
Department of Mathematics, 
Oskar-Morgenstern-Platz 1
A-1090 Vienna
} 
\email{martin.ehler@univie.ac.at}
\begin{abstract}
We consider signal reconstruction from the norms of subspace
components generalizing standard phase retrieval problems. In the
deterministic setting, a  closed reconstruction formula is derived when
the subspaces satisfy certain cubature conditions, that require at
least a quadratic number of subspaces.
Moreover, we address reconstruction under the erasure of a subset of the
norms; using the concepts of $p$-fusion frames and list decoding, we
propose an algorithm 
that outputs a finite list of candidate signals, one of which is the
correct one. 
In the random setting, we show that a set of subspaces chosen at
random and of cardinality
scaling  linearly in the ambient dimension allows for exact
reconstruction with high probability  by solving the feasibility problem of a semidefinite program.
\end{abstract}



\maketitle

\section{Introduction}

%

The phase retrieval problem, which refers to the task of recovering a
signal from the absolute values of linear measurements, has received
much attention recently: see
\cite{Balan:2009fk,Balan:2006fk,Candes:2011fk,Candes:uq,Chai:2011fk,Waldspurger:2012uq}
to mention only few. We are dealing with a generalization, in which
the measurements consist of norms of projections of the signal onto
$k$-dimensional subspaces. For $k=1$, our setting reduces to the
classical phase retrieval problem. 

%
%
%
%



Here, we pose the following questions: Under which properties of the subspaces can we reconstruct the original signal from the norms of its $k$-dimensional subspace components by means of a closed formula? Also, given that requiring a closed formula for reconstruction is too costly, can we develop strategies to reduce the number of required subspace components, under a numerical reconstruction?  We shall provide positive xanswers for a deterministic choice and a random choice of subspaces.

\smallskip
\noindent
\textbf{Deterministic setting:} 
Given $k$-dimensional linear subspaces $\{V_j\}_{j=1}^n$ in $\R^d$, we
aim to reconstruct the signal $x\in \R^d$ from
$\{\|P_{V_j}(x)\|\}_{j=1}^n$, where $P_{V_j}$ denotes the orthogonal
projector onto $V_j$. Clearly, $x$ can only be recovered up to its sign. 
In \cite{Cahill:2013fk}, several characterizations of subspaces such
that the mapping $\{\pm x\}\mapsto \{\|P_{V_j}(x)\|\}_{j=1}^n$ is injective
are given, see also \cite{Bahmanpour:bf}. Our aim is to showcase properties of the subspaces that
moreover allow for an explicit reconstruction formula. 

If there are positive weights 
$\{\omega_j\}_{j=1}^n$ such that $\{(V_j,\omega_j)\}_{j=1}^n$ yields a so-called cubature of strength $4$ as defined in Section \ref{sec:cub} of the present paper,  then we shall obtain a closed reconstruction formula for $xx^*$ enabling us to extract $\pm x$. Thus, we extend the $1$-dimensional results in \cite{Balan:2009fk} to $k$-dimensional projections. Note that the authors in \cite{Balan:2009fk} require cubatures for the projective space whose weights are $\omega_j=1/n$, i.e., so-called projective designs. In practice, however, the choice of subspaces may underlie restrictions that prevent them from being a design. Therefore, our results are a significant improvement for $1$-dimensional projections already. 

To address subspace erasures, we suppose that we are only given the values of $n-p$ norms 
and we need to reconstruct the missing $p$ norms. Notice that our input are not the subspace components but their norms, as opposed to signal reconstruction under the erasures discussed in \cite{Bodmann:2005uq,Holmes:2004fk,Kutyniok:2009aa}. If there are positive weights $\{\omega_j\}_{j=1}^n$ such that $\{(V_j,\omega_j)\}_{j=1}^n$ forms a tight $p$-fusion frame as recently introduced in \cite{Bachoc:2010aa}, then the computation of the erased norms up to permutations  amounts to solving a system of algebraic equations. We can then reconstruct $\pm x$ from the entire set of $n$ magnitude subspace components. 
In other words, we found conditions on subspaces, so that we can compute a finite list of candidate signals, one of which is the correct one. The latter is a form of list decoding as introduced in \cite{Elias:1991fk}.

\smallskip
The limit of this deterministic approach stands in the required number
of subspaces. Indeed, it is known that the cardinality of a cubature
formula of strength $4$ scales at least quadratically with the ambient
dimension $d$. In the random setting, it will be possible to reduce
the number of subspaces
to linear size:

\smallskip
\noindent
\textbf{Random setting:} We shall extend to $k$-dimensional subspaces
the results obtained for $k=1$ in a recent series of papers
\cite{Candes:uq, Demanet:2012uq, Candes:2012fk}.
In  \cite{Candes:uq} it was shown  that semidefinite programming
yields signal recovery with high probability when the $1$-dimensional
subspaces are chosen at random and that the cardinality of the
subspaces can scale linearly in the ambient dimension up to a
logarithmic factor. Numerical stability in the presence of noise was also verified. 
The underlying semidefinite program was shown in \cite{Demanet:2012uq}
to afford (with high probability) a unique feasible solution, and the logarithmic factor was removed in \cite{Candes:2012fk}. 

Our proof for $k$-dimensional subspaces  (see Theorem
\ref{th:finale!}) follows the approach in
\cite{Candes:2012fk,Candes:uq}. We verify that randomly selected subspaces satisfy a near isometry property and ensure the existence of a so-called dual certificate, which implies that the solution of the semidefinite program indeed recovers the signal with high probability. However, the generalization to $k$-dimensional projections raised additional difficulties. Indeed, the case $k=1$ relies on
random vectors whose entries are i.i.d.~Gaussian modeling the
measurements. For $k>1$, we must deal with measurement matrices having
orthogonal rows, so that entries from one row are stochastically
dependent on those in any other row. Hence, the extension from $k=1$
to $k>1$ is not obvious and requires special care. In particular, we apply weak convergence results from random matrix theory to derive the lower estimate for the near isometry property, see our Proposition \ref{lemma:ck}. Moreover, we present numerical experiments to illustrate the practical feasibility of the method in small dimensions. 

\smallskip
\noindent
\textbf{Complex case:}
Although we present our results for real signals and subspaces exclusively, the agenda can also be followed in the complex setting. We shall discuss the required modifications at the end of the present paper.

\bigskip

We would like to mention that signal reconstruction from phaseless measurements is a common problem in optical physics such as X-ray crystallography and diffraction imaging, where coherent light sources correspond to magnitude measurements of Fourier frame coefficients. Crystal twinning \cite{Drenth:2010fk}, on the other hand, involves signal reconstruction from averaged diffraction patterns by means of incoherent addition of $k$ wavefields, where usually $k=1,2,3$ \cite{Elser:2008fk}. Indeed, incoherent light sources involve (weighted) sums of $k$ squared moduli of Fourier coefficients. The latter is essentially the squared norm of the orthogonal projection onto the associated $k$-dimensional subspace. Hence, our mathematical setting of rank-$k$ orthogonal projectors, where $k$ is independent of the ambient dimension $d$, relates to measurements in optical physics. Nonetheless, we should mention that we do not focus on Fourier type measurements and do not incorporate any add
 itional information and side constraints that are commonly available in optical physics measurements and that are used in standard reconstruction algorithms, see \cite{Fienup:1982vn,Gerchberg:1972kx} and also \cite{Bauschke:2002ys}.

%


\smallskip
\noindent
\textbf{Outline:}
In Section \ref{sec:1}, we recall fusion frames, state the phase retrieval problem, and introduce tight $p$-fusion frames and cubature formulas. We present the closed  reconstruction formula in Section \ref{sec:exact} and  our reconstruction algorithm 
in presence of erasures in Section \ref{sec:3}. The random subspace selection is addressed in Section \ref{sec:prefinal}. Numerical experiments are presented in Section \ref{sec:final}, and we discuss the complex setting in Section \ref{sec:complex}.

\section{Fusion frames, phase retrieval, and cubature formulas}\label{sec:1}
\subsection{Fusion frames and the problem of reconstruction without phase}
Let $\mathcal{G}_{k,d}=\mathcal{G}_{k,d}(\R)$ denote the \emph{real Grassmann space}, i.e., the $k$-dimensional subspaces of $\R^d$. Each $V\in \mathcal{G}_{k,d}$ can be identified with the orthogonal projector onto $V$, denoted by $P_V$. Let $\{V_j\}_{j=1}^n\subset \mathcal{G}_{k,d}$ and let $\{\omega_j\}_{j=1}^n$ be a collection of positive weights. Then $\{(V_j,\omega_j)\}_{j=1}^n$ is called a \emph{fusion frame} if there are positive constants $A$ and $B$ such that
\begin{equation}\label{eq:fusion def}
A\|x\|^2 \leq \sum_{j=1}^n \omega_j\|P_{V_j}(x)\|^2 \leq B \|x\|^2, \text{ for all $x\in\R^d$,}
\end{equation}
cf.~\cite{Casazza:2008aa}. The condition \eqref{eq:fusion def} is equivalent to 
\begin{equation}\label{eq:tight fusion def}
A\leq \sum_{j=1}^n \omega_j\langle P_{x},P_{V_j}\rangle \leq B, \text{ for all $x\in S^{d-1}$,}
\end{equation}
where $P_x$ is short for $P_{x\R}$ and $\langle P_{x},P_{V_j}\rangle:=\trace(P_xP_{V_j})$ is the standard inner product between self-adjoint operators. If $A=B$, then $\{(V_j,\omega_j)\}_{j=1}^n$ is called a \emph{tight fusion frame}, and any signal $x\in S^{d-1}$ can be reconstructed from its subspace components by the simple formula 
\begin{equation}\label{eq:tight reconstr}
x = \frac{1}{A}\sum_{j=1}^n  \omega_j P_{V_j}(x). 
\end{equation}
If, however, instead of $\{P_{V_j}(x)\}_{j=1}^n$ we only observe the norms $\{\|P_{V_j}(x)\|\}_{j=1}^n$ and, worse, we even lose some of these norms, can we still reconstruct $x$? Clearly, $x$ can be determined up to its sign at best. In the present paper, we find conditions on $\{(V_j,\omega_j)\}_{j=1}^n$ together with a computationally feasible algorithm that enable us to determine $\pm x$. 

\begin{remark}
We want to point out that $1$-bit compressed sensing, cf.~\cite{Boufounos:2008qd,Plan:2011jk}, deals with a problem that is complementary to phase retrieval. There, the magnitudes are unknown and signals are reconstructed from the signs of the frame coefficients. 
\end{remark}

\subsection{Tight $p$-fusion frames}

Let $\{V_j\}_{j=1}^n\subset \mathcal{G}_{k,d}$ and let $\{\omega_j\}_{j=1}^n$ be a collection of positive weights and $p$ a positive integer. Then $\{(V_j,\omega_j)\}_{j=1}^n$ is called a \emph{$p$-fusion frame} in \cite{Bachoc:2010aa} if there exist positive constants $A_p$ and $B_p$ such that
\begin{equation}\label{eq:p fusion def}
A_p\|x\|^{2p}\leq \sum_{j=1}^n \omega_j\|P_{V_j}(x)\|^{2p} \leq B_p\|x\|^{2p}, \text{ for all $x\in\R^d$,}
\end{equation}
see also \cite{Okoudjou:2010aa} for related concepts. 
If $A_p=B_p$, then $\{(V_j,\omega_j)\}_{j=1}^n$ is called a \emph{tight $p$-fusion frame}. As with \eqref{eq:fusion def} and \eqref{eq:tight fusion def}, the condition \eqref{eq:p fusion def} is equivalent to 
\begin{equation}\label{eq:p fusion def2}
A_p\leq \sum_{j=1}^n \omega_j\langle P_x, P_{V_j}\rangle^{p} \leq B_p, \text{ for all $x\in S^{d-1}$.}
\end{equation}
If $\{(V_j,\omega_j)\}_{j=1}^n$ is a tight $p$-fusion frame, then it is also a tight $\ell$-fusion frame for all integers $1\leq \ell\leq p$, and the tight $\ell$-fusion frame bounds are 
\begin{equation}\label{eq:new bound formula}
A_\ell= \frac{(k/2)_\ell}{(d/2)_\ell}\sum_{j=1}^n\omega_j,
\end{equation}
where we used $(a)_\ell=a(a+1)\cdots (a+\ell-1)$, cf.~\cite{Bachoc:2010aa}. We also refer to \cite{Bachoc:2010aa} for constructions and general existence results. 

\subsection{Cubature formulas}\label{sec:cub}
The real orthogonal group $O(\R^d)$ acts transitively on $\mathcal{G}_{k,d}$, and the Haar measure on $O(\R^d)$ induces a probability measure $\sigma_k$ on $\mathcal{G}_{k,d}$. Let $L^2(\Gkd)$ denote the complex valued functions on $\mathcal{G}_{k,d}$, whose squared module is integrable with respect to $\sigma_k$. The complex irreducible representations of $O(\R^d)$ are associated to partitions
$\mu=(\mu_1,\dots,\mu_d)$, $\mu_1\geq \ldots\geq \mu_d\geq 0$, denoted by 
$V_d^{\mu}$, cf.~\cite{Goodman:1998fk}. Let $l(\mu)$ be the number of nonzero entries in $\mu$ so that
\begin{equation}\label{dec L2}
L^2(\Gkd)=\bigoplus_{ l(\mu)\leq k} H_{k,d}^{2\mu}, \quad\text{ where } H_{k,d}^{2\mu} \simeq V_d^{2\mu},
\end{equation}
see \cite{Goodman:1998fk}. 
The space of polynomial functions on $\mathcal{G}_{k,d}$ of degree bounded by $2p$ is 
\begin{equation}\label{eq:polys}
\Pol_{\leq 2p} (\Gkd):=\bigoplus_{l(\mu)\leq k,\ |\mu|\leq p} H_{k,d}^{2\mu},
\end{equation}
and we additionally define the subspace 
\begin{equation}\label{eq:polys diag}
\Pol^1_{\leq 2p} (\Gkd):=\bigoplus_{l(\mu)\leq 1,\ |\mu|\leq p} H_{k,d}^{2\mu}.
\end{equation}
These spaces are explicitly given by
\begin{align*}
\Pol_{\leq 2p}(\Gkd) & = \spann \{V\mapsto \langle  P_{x_1},P_V\rangle \cdots\langle  P_{x_p},P_V\rangle : x_1,\ldots,x_p\in S^{d-1} \},\\
\Pol^1_{\leq 2p}(\Gkd) & = \spann \{V\mapsto \langle  P_{x},P_V\rangle^p : x\in S^{d-1} \},
\end{align*}
cf.~\cite[Remark 5.4, proof of Theorem 5.3]{Bachoc:2010aa}. 
Let $\{V_j\}_{j=1}^n\subset\Gkd$ and $\{\omega_j\}_{j=1}^n$ be a collection of positive weights normalized such that $\sum_{j=1}^n \omega_j=1$. Then $\{(V_j,\omega_j)\}_{j=1}^n$ is called a \emph{cubature of strength $2p$ for $\mathcal{G}_{k,d}$} if
\begin{equation}\label{def cubature formula}
\int_{\Gkd} f(V)d\sigma_k(V)= \sum_{j=1}^n \omega_j f(V_j) \quad \text{ for all } f\in \Pol_{\leq 2p} (\Gkd).
\end{equation}
Grassmannian designs, i.e., cubatures with constant weights, have been studied in \cite{Bachoc:2005aa,Bachoc:2006aa,Bachoc:2004fk,Bachoc:2002aa}. For existence results on cubatures and the relations between $p$ and $n$, we refer to \cite{Harpe:2005fk}. 
It was verified in \cite{Bachoc:2010aa} that $\{(V_j,\omega_j)\}_{j=1}^n$ is a tight $p$-fusion frame if and only if
\begin{equation*}
\int_{\Gkd} f(V)d\sigma_k(V)= \sum_{j=1}^n \omega_j f(V_j) \quad \text{ for all } f\in \Pol^1_{\leq 2p} (\Gkd).
\end{equation*}
Thus, any cubature of strength $2p$ is a tight $p$-fusion frame. The converse implication does not hold in general except for $p$ or $k$ equals $1$.

\begin{remark}
Note that the case $k=1$ with constant weights corresponds to projective designs. Spherical designs have been widely studied in the literature \cite{Bannai:1979kx,Delsarte:1977aa,Seidel:2001aa} and any antipodal spherical $2p$-design induces a projective $2p$-design by choosing the lines along the antipodal points.  
\end{remark}

\section{Signal reconstruction in the case of a cubature of strength $4$}\label{sec:exact}

Let $\mathscr{H}$ denote the collection of symmetric matrices in $\R^{d\times d}$. If $\{P_{V_j}\}_{j=1}^n$ spans $\mathscr{H}$, then standard results in frame theory imply that $S:\mathscr{H}\rightarrow\mathscr{H}$ given by $X\mapsto\sum_{j=1}^n\langle X,P_{V_j}\rangle P_{V_j}$ is invertible and 
\begin{equation*}
xx^* = \sum_{j=1}^n \|P_{V_j}(x)\|^2 S^{-1}(P_{V_j}), \quad \text{for all $x\in\R^d$.}
\end{equation*}
By imposing stronger conditions on $\{P_{V_j}\}_{j=1}^n$, the operator $S$ can be inverted explicitly. To that end, we establish the following result that generalizes the case $k=1$ treated in \cite{Balan:2009fk}. We point out that we allow for cubatures as opposed to projective designs in \cite{Balan:2009fk} that require the cubature weights to be constant:
\begin{proposition}\label{prop:2}
Let $\{(V_j,\omega_j)\}_{j=1}^n$ be a cubature of strength $4$ for $\mathcal{G}_{k,d}$. If $x\in S^{d-1}$, then 
\begin{equation}\label{eq:recon Px}
P_x=a_1\sum_{j=1}^n \omega_j \|P_{V_j}(x)\|^2 P_{V_j} - a_2I,
\end{equation}
where $a_1=\frac{d(d+2)(d-1)}{2k(d-k)}$ and $a_2=\frac{kd+k-2}{2(d-k)}$. 
\end{proposition}
\begin{proof}
For any $x,y\in S^{d-1}$, the function
$V\mapsto \langle P_x,P_V\rangle \langle P_y,P_V\rangle $ belongs to $\Pol_{\leq 4}(\Gkd)$. Applying the cubature formula yields 
\begin{equation}\label{e1}
\sum_{j=1}^n \omega_j \langle P_x,P_{V_j}\rangle \langle P_y,P_{V_j}\rangle = \int_{\Gkd} \langle P_x,P_{V}\rangle \langle P_y,P_{V}\rangle d\sigma_k(V).
\end{equation}
The function
\begin{equation}
G: (\R x,\R y)\mapsto  \int_{\Gkd} \langle P_x,P_{V}\rangle \langle P_y,P_{V}\rangle d\sigma_k(V)
\end{equation}
belongs to $L^2({\mathcal G}_{1,d}\times {\mathcal G}_{1,d})$ and is zonal. For each variable, it has the form $ \R x\mapsto \langle P_x,A(y)\rangle$, where $A(y)=\int_{\mathcal{G}_{k,d}} \langle P_y,P_V\rangle P_V d\sigma_k(V)$, and $ \R y\mapsto \langle P_y,A(x)\rangle$, respectively. Since $A(y)$ is self-adjoint and hence a linear combination of projections, $G(\cdot,\R y)$ and $G(\R x,\cdot)$ belong to $\Pol_{\leq 2}({\mathcal G}_{1,d})$. The zonal functions on the projective space are polynomials in the variable $\langle P_x, P_y\rangle=(x,y)^2$, so that $G$ must be of the form $\alpha_1 ( x,y)^2+\alpha_2$. Thus, \eqref{e1} yields 
\begin{equation}\label{e2}
\sum_{j=1}^n \omega_j \langle P_x,P_{V_j}\rangle \langle P_y,P_{V_j}\rangle = \alpha_1\langle P_x,P_y\rangle  + \alpha_2 \langle I, P_y\rangle .
\end{equation}
Since \eqref{e2} holds for every $y$, we derive 
\begin{equation}\label{e3}
\sum_{j=1}^n \omega_j \langle P_x,P_{V_j}\rangle P_{V_j}= \alpha_1 P_x + \alpha_2 I.
\end{equation}
Taking traces in \eqref{e3} leads to
$ 
k\sum_{j=1}^n \omega_j \langle P_x,P_{V_j}\rangle = \alpha_1  + d\alpha_2,
$ 
and the property of tight $1$-fusion frames gives
$ 
\sum_{j=1}^n \omega_j \langle P_x,P_{V_j}\rangle = A_1= k/d,
$ 
so we obtain 
\begin{equation}\label{eq:first fro alpha and beta}
\alpha_1+d\alpha_2= k^2/d.
\end{equation} 
Taking $x=y$ in \eqref{e2} implies
$ 
\sum_{j=1}^n \omega_j \langle P_x,P_{V_j}\rangle ^2= \alpha_1  + \alpha_2,
$ 
and the tight $2$-fusion frame property leads to 
$ 
\sum_{j=1}^n \omega_j \langle P_x,P_{V_j}\rangle ^2= A_2= k(k+2)/(d(d+2)),
$ 
so that we obtain 
\begin{equation}\label{eq:second alpha and beta}
\alpha_1+\alpha_2= k(k+2)/(d(d+2)).
\end{equation}
Solving for $\alpha_1$ and $\alpha_2$ in \eqref{eq:first fro alpha and beta} and \eqref{eq:second alpha and beta} yields the required identity with $a_1=1/\alpha_1$ and $a_2=\alpha_2/\alpha_1$.
\end{proof}

\begin{remark}\label{remark:13} Since any $X\in\mathscr{H}$ can be written as a sum of weighted orthogonal projectors, \eqref{eq:recon Px} can be extended to
\begin{equation}\label{eq:trace reconstruction perfect finite}
X=a_1\sum_{j=1}^n \omega_j \langle X,P_{V_j}\rangle P_{V_j} - a_2\trace(X) I.
\end{equation}
For $x\in\R^d$ and $X=xx^*$, the tight-$1$ fusion frame property yields $\trace(X)=\|x\|^2=\frac{d}{k}\sum_{j=1}^n \omega_j\|P_{V_j}(x)\|^2$, so that the entire right-hand side of \eqref{eq:trace reconstruction perfect finite} can be computed from $\{\|P_{V_j}(x)\|^2\}_{j=1}^n$ and hence $\pm x$ can be recovered. 

We can conclude from \eqref{eq:trace reconstruction perfect finite} that $\{\omega_jP_{V_j}\}_{j=1}^n$ and $\{Q_j\}_{j=1}^n$, where $Q_j=a_1P_{V_j}-a_2\frac{d}{k}I$, are pairs of dual frames for $\mathscr{H}$, i.e., 
\begin{equation*}
X = \sum_{j=1}^n  \langle X,\omega_jP_{V_j}\rangle Q_j,\quad\text{for all $X\in\mathscr{H}$.}
\end{equation*}
Moreover, if $V$ is a random subspace, uniformly distributed in $\mathcal{G}_{k,d}$, i.e., distributed according to $\sigma_k$, then the proof of Proposition \ref{prop:2} yields that 
\begin{equation}\label{eq:expec}
a_1\mathbb{E}(\langle X,P_{V}\rangle P_{V}) - a_2\trace(X) I = X, 
\end{equation}
for all $X\in\mathscr{H}$. Thus, if $\{V_j\}_{j=1}^n\subset\mathcal{G}_{k,d}$ are independent copies of $V$, then the law of large numbers implies 
\begin{equation}\label{eq:law ln}
\frac{a_1}{n}\sum_{j=1}^n  \langle X,P_{V_j}\rangle P_{V_j} - a_2\trace(X) I  \rightarrow X \quad \text{almost surely.}
\end{equation} 
However, $n$ must be chosen large to obtain an accurate representation of $X$. In Sections \ref{sec:prefinal} and \ref{sec:final}, we shall see that the random choice of subspaces can be efficient when the algebraic reconstruction formula is replaced with a semidefinite program. 
\end{remark}

\section{Algorithm for signal reconstruction from magnitudes of incomplete subspace components }\label{sec:3}

In this section, we consider the situation where $x\in S^{d-1}$, and we aim to reconstruct $\pm x$ from any $n-p$ elements of the set $\{\|P_{V_j}(x)\|^2\}_{j=1}^n$. Indeed, for fixed $p$, we are aiming at a reconstruction scheme valid for any subset of $p$ missing norms. Without loss of generality, we can assume that the first $p$ norms have been erased, so we want to recover $\pm x$ from the knowledge of 
$\{\|P_{V_j}(x)\|\}_{j=p+1}^n$.

In a first step, we attempt to compute the missing values 
\begin{equation*}
t_{j}:=\Vert P_{V_{j}}(x)\Vert^2,\quad 1\leq j\leq p. 
\end{equation*}
This will be made possible by the property that  $\{(V_j,\omega_j)\}_{j=1}^n$ is a tight $p$-fusion frame with $\sum_{j=1}^n\omega_j =1$.
The second step is dedicated to reconstructing $\pm x$ from $\{\|P_{V_j}(x)\|^2\}_{j=1}^n$.

\subsection{Step 1: reconstruction of the erased norms}\label{subsec:1}
The tight $p$-fusion frame $\{(V_j,\omega_j)\}_{j=1}^n$ is also a tight $\ell$-fusion frame for $1\leq \ell \leq p$, cf.~\cite[Proposition 5.1]{Bachoc:2010aa}, so that \eqref{eq:new bound formula} yields  
\begin{equation*}
\sum_{j=1}^n \omega_j\|P_{V_j}(x)\|^{2\ell}=A_{\ell}=\frac{(k/2)_{\ell}}{(d/2)_{\ell}},\quad   1\leq\ell\leq p.
\end{equation*}
Therefore,  $(t_1,\dots, t_p)$ is a solution of the algebraic system of equations
\begin{equation*}\label{eq:algebraic}
\sum_{j=1}^p \omega_j T_j^{\ell}=\frac{(k/2)_{\ell}}{(d/2)_{\ell}}-\sum_{j=p+1}^n \omega_j \Vert P_{V_j}(x)\Vert^{2\ell} ,\quad   1\leq\ell\leq p, \tag{AE}
\end{equation*}
in the unknowns $(T_1,\ldots,T_p)$. To start with, let us consider the special case of equal weights; then, \eqref{eq:algebraic} gives the values of the symmetric powers $\sum_{j=1}^p t^\ell_j$, for $\ell=0,\ldots,p$, which, as polynomial expressions, generate the ring of symmetric polynomials up to degree $p$. Vieta's formula yields
\begin{equation*}
\prod_{i=1}^p (T-t_i) = \sum_{j=0}^p (-1)^j e_j T^{p-j}, \quad\text{where $e_0=1$ and } e_j=\sum_{1\leq i_1<\ldots<i_j\leq p} t_{i_1}\cdots t_{i_p},
\end{equation*}
and Newton's identity leads to
\begin{equation*}
e_j=\frac{1}{j}\sum_{\ell=1}^j (-1)^{\ell-1}e_{j-\ell}\sum_{j=1}^p t^\ell_j,\quad\text{for $j=1,\ldots,p$.}
\end{equation*}
Therefore, we can compute the coefficients of $\prod_{i=1}^p (T-t_i)$ as a polynomial in $T$ and solve for its roots; we see that $(t_1,\dots, t_p)$ is determined up to a permutation so that we obtain at most $p!$ distinct solutions to \eqref{eq:algebraic}.

If the weights are not equal, one can still show that \eqref{eq:algebraic} has at most $p!$ solutions. The issue is to verify that the associated affine variety is zero-dimensional. Results from intersection theory and the refined B{\'e}zout theorem, cf.~\cite{Fulton:1984fk,Sottile:2001uq} and \cite{Bochnak:1998kx}, then imply that the variety's cardinality is at most the product of the degrees of the $p$ polynomials, i.e., there are at most $p!$ solutions:
\begin{proposition}\label{Prop:algebraic system}
Let $\{b_\ell\}_{\ell=1}^p$ be complex numbers and define
\begin{equation*}
f_\ell(T) =  \sum_{j=1}^p \omega_j T_j^\ell -b_\ell,\quad \ell=1,\ldots,p.
\end{equation*}
If $\{\omega_j\}_{j=1}^p$ are positive numbers, then the affine variety $\mathcal{V}:=\{T\in\C^p : f_1(T)=0,\ldots,f_p(T)=0\}$ is zero-dimensional. 
\end{proposition}
\begin{proof}
We proceed by induction on $p$. The assertion is certainly true for $p=1$. Next, we observe that the Jacobian determinant satisfies
\begin{align*}
\det\big( \frac{\partial (f_1,\ldots,f_p)}{\partial(T_1,\ldots,T_p)} \big) & = \det \begin{pmatrix}  \omega_1 & \hdots & \omega_p \\
\vdots & & \vdots\\
\omega_1 p T_1^{p-1} & \hdots & \omega_1 p T_p^{p-1}
\end{pmatrix}.
\intertext{By applying the multilinearity of the determinant once to the columns and another time to the rows, we obtain}
\det\big( \frac{\partial (f_1,\ldots,f_p)}{\partial(T_1,\ldots,T_p)} \big) 
& = \omega_1\cdots\omega_p \cdot p!\cdot \det\begin{pmatrix}  1& \hdots & 1 \\
\vdots & & \vdots\\
 T_1^{p-1} & \hdots & T_p^{p-1}
\end{pmatrix}.
\intertext{The well-known formula for the Vandermonde determinant yields}
\det\big( \frac{\partial (f_1,\ldots,f_p)}{\partial(T_1,\ldots,T_p)} \big) 
& = \omega_1\cdots\omega_p \cdot p!\cdot \prod_{1\leq i<j\leq p} (T_j-T_i).
\end{align*}
For $i<j$, let $\Delta_{i,j}:=\{T\in\C^d : T_i=T_j\}$ denote the diagonals. The Jacobian determinant is apparently nonzero for $T\not\in \bigcup_{i<j}\Delta_{i,j}$. Therefore, every $T\in \mathcal{V}\setminus \bigcup_{i<j}\Delta_{i,j}$ is a nonsingular point of $\mathcal{V}$, and the dimension of $\mathcal{V}$ at $T$ is $p-p=0$, cf.~\cite[Theorem 9.9]{Cox:1996fk} and \cite[Lemma 11.5.1]{Bochnak:1998kx}. 
It remains to consider the intersection of $\mathcal{V}$ with $\Delta_{i,j}$. 
To fix ideas, let us consider the case $i = 1$, $j = 2$. The intersection $\mathcal{V} \cap\Delta_{1,2}$ is given by the system of equations
\begin{equation*}
(\omega_1+\omega_2)T_2^\ell+\sum_{j=3}^p \omega_jT_j^\ell = b_\ell ,\quad \ell=1,\ldots,p.
\end{equation*}
Because $\omega_1+\omega_2>0$, by induction the first $p-1$ of these equations have only finitely many solutions. Thus, $\mathcal{V}\cap\Delta_{1,2}$ is finite, too.
\end{proof}

\begin{remark}
The above proof shows that the positivity assumption on the weights in Proposition \ref{Prop:algebraic system} can be replaced with $\omega_{j_1}+\ldots+\omega_{j_i}\neq 0$, for all $1\leq j_1<\ldots<j_i\leq p$ and $i=1\ldots,p$.
\end{remark}

We have proved that the system of algebraic equations \eqref{eq:algebraic} has at most $p!$ complex solutions. In order to compute these solutions,
standard algorithmic methods can be applied \cite{Cohen:1999xx, Cox:1996fk}. The construction of a Gr\"obner basis of the ideal $\mathcal I$ generated by the $p$ equations allows
to compute the algebraic operations in the quotient  ring $\R[T_1,\dots,T_p]/{\mathcal I}$, which is finite dimensional and of dimension at most $p!$. The computation of the solutions 
then boils down to linear algebra in this space.

%
%
%

\subsection{Step 2: reconstruction from the magnitude of subspace components}\label{section:step 2}
In this second step, we try to compute $P_x$ from each of the possible candidates for $\{\| P_{V_j}(x)\|^2\}_{j=1}^n$ derived 
from  a solution $(t_1,\dots,t_p)$ of \eqref{eq:algebraic}. For this, we assume that $\{(V_j,\omega_j)\}_{j=1}^n$ is also a cubature of strength $4$,
and we  apply formula \eqref{eq:recon Px} where we replace $\| P_{V_j}(x)\|^2$ by $t_j$ for $1\leq i\leq p$.

To summarize, we have proved:
\begin{theorem}\label{th:all}
Let $\{(V_j,\omega_j)\}_{j=1}^n$ be a tight $p$-fusion frame that is also a cubature of strength $4$ for $\mathcal{G}_{k,d}$. 
If $x\in S^{d-1}$, then Algorithm 1 outputs a list $L$ of at most $2p!$ elements of $S^{d-1}$ containing $x$.
\begin{algorithm}
\caption{List reconstruction}
\begin{algorithmic}[1]
\REQUIRE $\{t_j:=\Vert P_{V_j}(x)\Vert^2\}_{j=p+1}^n$.
\ENSURE $L$, $x\in L$.
\STATE Initialize $L=\emptyset$.
\STATE Compute the set $\mathcal S$ of solutions of the algebraic system of equations in the unknowns $T_1,\dots, T_p$:
\begin{equation*}
\sum_{j=1}^p \omega_j T_j^{\ell}=\frac{(k/2)_{\ell}}{(d/2)_{\ell}}-\sum_{j=p+1}^n 
\omega_j t_j^{\ell} ,\quad   1\leq\ell\leq p.\tag{AE}
\end{equation*}
\STATE For every $(t_1,\dots, t_p)\in \mathcal S$, and $\alpha,\beta$ defined in Proposition \ref{prop:2}, compute
\begin{equation*}
P=a_1\sum_{j=1}^n \omega_j t_j P_{V_j} - a_2 I.
\end{equation*} 
\STATE If $P$ is a projection of rank $1$, compute a unit vector $\xi$ spanning its image
and add $\pm \xi$ to $L$.
\RETURN $L$
\end{algorithmic}
\end{algorithm}
\end{theorem}

Note that the cardinality of the list $L$ is triggered by the number of erasures $p$. The number of measurements depends on $p$ and on the ambient signal dimension $d$. 
The weighted subspaces are supposed to form a cubature of strength $4$ for $\mathcal{G}_{k,d}$, therefore, we must have at least $n\geq \frac{1}{2}d(d+1)$ many subspaces, see \cite{Harpe:2005fk}. Hence, the cardinality scales at least quadratically in the ambient dimension $d$ already for $p=2$. In general, the minimal number of measurements is a growing function of $p$ and $d$. We refer to \cite{Bachoc:2010aa} for some explicit constructions of tight $p$-fusion frames. 

We also note that the actual output list $L$ can be much shorter than $2p!$ because many solutions of the algebraic system of equations will not lead to a candidate for the signal $x$.
In the first place, we can exclude those solutions of \eqref{eq:algebraic} that are not real or have negative entries. Moreover, one can expect that, for most solutions of \eqref{eq:algebraic}, the symmetric operator $P$ in step 3 is not a rank-one projector. Also, the solutions $(t_1,\dots,t_p)$ of \eqref{eq:algebraic} that do not satisfy
$|t_i^{1/2}-t_j^{1/2}|^2\leq \|P_{V_i}-P_{V_j}\|^2$, for every $1\leq i<j\leq n$, 
can be removed  because they  violate the consistency conditions
\begin{equation*}
|\|P_{V_i}(x)\|-\|P_{V_j}(x)\||^2 \leq  \|P_{V_i}(x)-P_{V_j}(x)\|^2 \leq \|P_{V_i}-P_{V_j}\|_\infty^2,
\end{equation*}
where $\|P_{V_i}-P_{V_j}\|_\infty$ denotes the operator norm of $P_{V_i}-P_{V_j}$. 


\begin{remark}\label{remark:1.1}
For $p=2$, the assumptions in Theorem \ref{th:all} reduce to $\{(V_j,\omega_j)\}_{j=1}^n$ being a cubature of strength $4$ for $\mathcal{G}_{k,d}$.  Even for $k=1$, our result extends \cite{Balan:2009fk} since we only need $n-2$ elements of the collection $\{\|P_{V_j}(x)\|^2\}_{j=1}^n$ as opposed to all $n$ elements in \cite{Balan:2009fk}. This additional flexibility is not for free: We must assume that $x\in S^{d-1}$, and, instead of the two possibilities $\pm x$ in \cite{Balan:2009fk}, we obtain a list $L$ of $4$ elements, one of which is $x$.  
%


%
\end{remark}

%

\section{Replacing the algebraic reconstruction formula with semidefinite programming}\label{sec:prefinal}
We assume in Proposition \ref{prop:2} that the weighted subspaces form a cubature of strength $4$ for $\mathcal{G}_{k,d}$. However, any real cubature of strength $4$ requires at least $\frac{1}{2}d(d+1)$ subspaces, see \cite{Harpe:2005fk}. Hence, the cardinality scales at least quadratically in the ambient dimension $d$. In this section, we replace the algebraic reconstruction formula with a feasibility problem of a semidefinite program similar to the approach in \cite{Candes:uq,Demanet:2012uq}, where the case $k=1$ was discussed. 

Recall that $\mathscr{H}$ denotes the collection of symmetric matrices in $\R^{d\times d}$. For $\{V_j\}_{j=1}^n\subset\mathcal{G}_{k,d}$, we define the operator 
\begin{equation}\label{eq:F}
\mathcal{F}_n : \mathscr{H} \rightarrow \R^n,\quad X\mapsto \frac{d}{k}(\langle X,P_{V_j}\rangle )_{j=1}^n. 
\end{equation} 
For $x\in \R^d$, let $f:=\frac{d}{k}(\|P_{V_j}(x)\|^2)_{j=1}^n=\mathcal{F}_n(xx^*)\in\R^n$, and we now aim to reconstruct $\pm x$ from $f$. By assuming that the union of the subspaces $\{V_j\}_{j=1}^n$ spans $\R^d$, clearly, $xx^*$ is a solution of 
\begin{equation}\label{eq:rank optimization}
\min_{X\in\mathscr{H}}( \rank(X)),\quad\text{subject to} \quad \mathcal{F}_n(X) = f,\; X\succeq 0.
\end{equation}
The notation $X\succeq 0$ stands for $X$ being positive semidefinite. Rank minimization is in general NP-hard, and in convex optimization it is standard to replace \eqref{eq:rank optimization} with 
\begin{equation}\label{eq:convex}
\min_{X\in\mathscr{H}}( \trace(X)),\quad\text{subject to} \quad \mathcal{F}_n(X) = f,\; X\succeq 0,
\end{equation}
a semidefinite program, for which efficient algorithms based on
interior point methods are available. The NEOS Server
\cite{Czyzyk:1998fk} provides online solvers for semidefinite programs. We know that the solution of \eqref{eq:rank optimization} has rank $1$, so there is more structure to it and, as in \cite{Demanet:2012uq}, we can consider the underlying feasibility problem, i.e., 
\begin{equation}\label{eq:convex II}
\text{find } X\in\mathscr{H}, \quad\text{subject to}\quad\mathcal{F}_n(X) = f,\; X\succeq 0.
\end{equation}
For $k=1$, there is a constant $c>0$, such that the random choice of at least $c d$ subspaces yields that, with high probability, $xx^*$ is the only solution to \eqref{eq:convex II}, i.e., the only feasible point of \eqref{eq:rank optimization} and \eqref{eq:convex}, cf.~\cite{Candes:2012fk,Candes:uq,Demanet:2012uq}. Here, we extend the result to $k>1$:
\begin{theorem}\label{th:finale!}
There are constants $c_1,c_2>0$ such that, if $n\geq c_1d$ and $\{V_j\}_{j=1}^n\subset\mathcal{G}_{k,d}$ are chosen independently identically distributed according to $\sigma_k$,  then, for all $x\in\R^d$, the matrix $xx^*$ is the unique solution to \eqref{eq:convex II} with probability at least $1-e^{-c_2 n}$. 
\end{theorem}
Note that the probability of exact recovery in Theorem \ref{th:finale!} holds simultaneously over all input signals $x\in\R^d$, and the constants  are independent of the ambient dimension $d$ but may depend on the subspace dimension $k$.

To verify Theorem \ref{th:finale!}, we shall first derive deterministic conditions serving uniqueness in \eqref{eq:convex II}. Later, we shall verify that these conditions are satisfied with high probability when the subspaces are chosen in the appropriate random fashion. After having assembled all ingredients, the proof of Theorem \ref{th:finale!} is presented in Appendix \ref{app:4}.

A simple rescaling allows us to restrict the considerations to $x\in S^{d-1}$. Let $T:=T_x:=\{xy^*+yx^*: y\in\R^d\}\subset \mathscr{H}$, and, for $Z\in\R^{d\times d}$, denote $Z_T$ its orthogonal projection onto $T$ and $Z_{T^\perp}$ its orthogonal projection onto the orthogonal complement of $T$. The term $\|\cdot\|_1$ denotes the nuclear norm and $\|\cdot\|_\infty$ the operator norm:
\begin{theorem}\label{th:candes}
Let $\{V_j\}_{j=1}^n\subset \mathcal{G}_{k,d}$ and $f=(\|P_{V_j}(x)\|^2)_{j=1}^n$. Assume that $0<A,B$ and $\gamma < A/B$ are fixed numbers, such that the following three points are satisfied:
\begin{itemize}
\item[\textnormal{(a)}] For all positive semidefinite matrices $X\in\mathscr{H}$, 
\begin{equation}\label{eq:1}
\frac{1}{n}\|\mathcal{F}_n(X)\|_{\ell_1} \leq B\|X\|_1.
\end{equation}
\item[\textnormal{(b)}] For all $X\in T$,
\begin{equation}\label{eq:2}
A\|X\|_\infty\leq \frac{1}{n}\|\mathcal{F}_n(X)\|_{\ell_1}.
\end{equation}
\item[\textnormal{(c)}] There exists $Y$ in the range of $\mathcal{F}_n^*$ such that
\begin{equation}\label{eq:3}
\|Y_T\|_1\leq \gamma  ,\qquad Y_{T^\perp}\succeq I_{T^\perp}.
\end{equation}
\end{itemize}
Then $xx^*$ is the unique solution to \eqref{eq:convex II}.
\end{theorem}
The matrix $Y$ in \eqref{eq:3} was called a dual certificate in \cite{Candes:uq}. 
To verify Theorem \ref{th:candes}, we can straightforwardly follow the lines of the proof in \cite{Candes:2012fk,Demanet:2012uq} while keeping track of the constants, see also \cite{Candes:uq,Chandrasekaran:2012fk}. The complete proof is in Appendix \ref{app:1}.

\begin{remark}
If $\{V_j\}_{j=1}^n$ is a design of strength $4$, then the conditions in Theorem \ref{th:candes} can be satisfied. Indeed, we can choose $B=1$ and there is $A_k>0$ satisfying \eqref{eq:2} that is even allowed to depend on $d$ in this case. Since $\mathcal{F}_n^*$ is onto, the certificate $Y=2I-2P_x$ is admissible and $\gamma$ can be zero. 
\end{remark}

In the subsequent sections, we shall verify that the conditions of Theorem \ref{th:candes} are satisfied with high probability when the subspaces $\{V_j\}_{j=1}^n$ are selected at random.

\subsection{Nuclear norm estimates on $\|\mathcal{F}_n(X)\|_{\ell_1}$ for $X\succeq 0$}\label{sec:5.1}
We shall verify that $\mathcal{F}_n$ is close to an isometry with high probability:
\begin{theorem}\label{th:singular}
Let $\{V_j\}_{j=1}^n\subset\mathcal{G}_{k,d}$ be independently chosen random subspaces with identical distribution $\sigma_k$. 
For $0<r<1$ fixed, there are constants $c(r),C(r)>0$, such that, for all positive semidefinite matrices $X$ and $n\geq c(r) d$,
\begin{equation}\label{eq:l1 norms}
(1-r)\|X\|_1 \leq \frac{1}{n}\|\mathcal{F}_n(X)\|_{\ell_1} \leq (1+r)\|X\|_1
\end{equation}
holds with probability at least $1-e^{-C(r) n}$.
\end{theorem}
By using the spectral decomposition of $X$, we see that condition \eqref{eq:l1 norms} is equivalent to 
\begin{equation*}
(1-r)\frac{nk}{d}\|x\|^2 \leq \sum_{j=1}^n \|P_{V_j}(x)\|^2 \leq (1+r)\frac{nk}{d}\|x\|^2, \quad\text{for all $x\in\R^d$}.
\end{equation*}
In other words, $\{V_j\}_{j=1}^n$ is a fusion frame that is not too far from being tight. It turns out that we can follow the lines in \cite{Candes:uq} to prove Theorem \ref{th:singular} after having established some analogy between $k=1$ and $k>1$. If $k=1$, the random variable $d\|P_{V}(x)\|^2$, for $x\in S^{d-1}$, is sub-exponential. We can verify the analogue result for $k>1$: 
\begin{lemma}\label{lemma:2xi}
If $V$ is a random subspace distributed according to $\sigma_k$ on $\mathcal{G}_{k,d}$, then, for any $x\in S^{d-1}$, 
\begin{equation}\label{eq:sub-exponential for beta}
\sup_{p\geq 1}p^{-1}\big(\mathbb{E}(\frac{d}{k}\|P_{V}(x)\|^2)^{p}\big)^{1/p}\leq 1.
\end{equation}
\end{lemma}
\begin{proof}[Proof of Lemma \ref{lemma:2xi}]
The distribution of $\|P_V(x)\|^2$ does not depend on the particular choice of $x\in S^{d-1}$ and is beta distributed with parameters $(\frac{k}{2},\frac{d-k}{2})$. Thus, its moments are given by 
\begin{equation}\label{eq:equalities for higher moments}
\mathbb{E}\|P_V(x)\|^{2p} = \frac{(k/2)_p}{(d/2)_p}=\frac{k(k+2)\cdots(k+2p-2)}{d(d+2)\cdots(d+2p-2)},
\end{equation}
which coincide with the tight $p$-fusion frame bounds \eqref{eq:new bound formula} when the weights are constant. An induction over $p$ yields \eqref{eq:sub-exponential for beta}.
\end{proof}
Note that Lemma \ref{lemma:2xi} says that $\frac{d}{k}\|P_{V}(x)\|^2$ is a sub-exponential random variable with a bound in \eqref{eq:sub-exponential for beta} that does not depend on $d$. The latter is one of the main ingredients to verify Theorem \ref{th:singular} along the lines in \cite{Candes:uq}, see Appendix \ref{app:2} for the details.

\subsection{Operator norm estimates on $\|\mathcal{F}_n(X) \|_{\ell_1}$ for symmetric rank-$2$ matrices}\label{sec:5.2}
We shall verify the condition \eqref{eq:2}:
\begin{theorem}\label{th:ck}
Let $k$ be fixed. There is a constant $u>0$ such that, for $0<r<1$ fixed, there exist constants $c,C>0$, such that, for all $n\geq c d$ and $\{V_j\}_{j=1}^n\subset\mathcal{G}_{k,d}$ independently chosen random subspaces with identical distribution $\sigma_k$,  the inequality 
 \begin{equation*}
\frac{1}{n}\|\mathcal{F}_n(X)\|_{\ell_1} \geq u (1-r)\|X\|_\infty,
\end{equation*}
for all symmetric rank-$2$ matrices $X$, holds with probability at least $1-e^{-C n} $. 
\end{theorem}
Note that the probability in the estimate in Theorem \ref{th:ck} is uniform in $X$. The proof of Theorem \ref{th:ck} is based on the following Proposition that was derived for $k=1$ in \cite{Candes:uq}. Our proof for $k>1$ is original:
\begin{proposition}\label{lemma:ck}
Let $k$ be fixed and $\{V_j\}_{j=1}^n\subset\mathcal{G}_{k,d}$ be independently chosen random subspaces with identical distribution $\sigma_k$. 
There is a constant $u>0$ such that, for all $-1\leq t\leq 1$ and $z_1,z_2\in S^{d-1}$ with $z_1\perp z_2$,
\begin{equation*}
\frac{d}{k}\mathbb{E}\big|\|P_V(z_1)\|^2-t\|P_V(z_2)\|^2\big| \geq u.
\end{equation*}
\end{proposition}
\begin{proof}
The sphere is two-point homogeneous and $\sigma_k$ is invariant under orthogonal transformation so that we can restrict the analysis to the first two canonical basis vectors $e_1$ and $e_2$. 
Since the integral is always nonzero, we only need to take care of the limit $d\rightarrow \infty$. We first see that 
\begin{align*}
\frac{d}{k}\mathbb{E}\big|\|P_V(e_1)\|^2-t\|P_V(e_2)\|^2\big| &  = \frac{d}{k}\int_{\mathcal{G}_{k,d}}\big|\|P_V(e_1)\|^2-t\|P_V(e_2)\|^2\big|d\sigma_k(V)\\
& = \frac{d}{k}\int_{\mathcal{V}_{2,d}}\big| \sum_{i=1}^k m^2_{i,1} - t\sum_{i=1}^k m^2_{i,2}\big|d\nu_2(M),
\end{align*}
where $\mathcal{V}_{2,d}=\{M=(m_{i,j})\in\R^{d\times 2}: M^*M=I \}$ denotes the Stiefel-manifold endowed with the standard probability measure $\nu_2$. If $M$ is a random matrix, distributed according to $\nu_2$, then, according to \cite[Proposition 7.5]{Eaton:1989fk}, the upper $k\times 2$ block of $M$ multiplied by $d$ converges in distribution (for $d\rightarrow \infty$) towards a random $k\times 2$ matrix whose entries are standard normal i.i.d.. Let us denote the underlying probability measure on $\R^{k\times 2}$ by $\mathcal{N}(0,I_k\otimes I_2)$. The convergence in distribution implies that, for $d\rightarrow \infty$,
\begin{equation*}
d\mathbb{E}\big|\|P_V(e_1)\|^2-t\|P_V(e_2)\|^2\big|   \rightarrow \int_{\R^{k\times 2}}\!\big| \|N(e_1)\|^2 - t\|N(e_2)\|^2\big|d\mathcal{N}(0,I_k\otimes I_2) (N).
\end{equation*}
Since the right-hand side is bigger than $0$, for all $-1\leq t\leq 1$, compactness and continuity arguments suffice to conclude the proof.
\end{proof}
For the complete proof of Theorem \ref{th:ck} that is based on Proposition \ref{lemma:ck}, we refer to Appendix \ref{app:3}.

\subsection{The dual certificate $Y$}
To derive the dual certificate $Y$, we can follow the ideas in \cite{Candes:2012fk,Candes:uq,Demanet:2012uq} adapted to $k>1$. We will use Proposition \ref{prop:2} from the deterministic setting and the Remark \ref{remark:13}. Let $\{V_j\}_{j=1}^n\subset\mathcal{G}_{k,d}$ be independently chosen random subspaces with identical distribution $\sigma_k$. 
The choice 
\begin{equation*}
Y_1:=2I-2P_x
\end{equation*}
would satisfy both conditions in \eqref{eq:3} but may not lie in the range of 
\begin{equation*}
\mathcal{F}_n^* : \R^n\rightarrow \mathscr{H},\qquad (\lambda_j)_{j=1}^n \mapsto \frac{d}{k}\sum_{j=1}^n \lambda_j P_{V_j}. 
\end{equation*}
Thus, we aim to determine an appropriate sequence $(\lambda_j)_{j=1}^n$ such that $\frac{d}{k}\sum_{j=1}^n \lambda_j P_{V_j}$ ``approximates'' $Y_1$. First, let us rewrite 
\begin{equation*}
Y_1=(k+2)I-(2P_x+kI).
\end{equation*}
For $ a:= \frac{2d(d-k)}{(d+2)(d-1)}$ and $b:= \frac{d(kd+k-2)}{(d+2)(d-1)}$, we observe that $a\rightarrow 2$ and $b\rightarrow k$ when $d$ tends to infinity, so that we can approximate $Y_1$ by
\begin{equation*}
Y_2: =   (k+2)I-(aP_x+bI).
\end{equation*}
Since Proposition \ref{prop:2} implies
\begin{equation}\label{eq:cubature representation}
aP_x+bI=\frac{d^2}{k}\mathbb{E} \|P_{V}(x)\|^2 P_{V}
\end{equation} 
and $\frac{d}{k}\mathbb{E}P_{V} = I$ holds, we obtain 
\begin{equation*}
Y_2 = \frac{d}{k}\mathbb{E}((k+2-d\|P_V(x)\|^2)P_V).
\end{equation*}
The sample mean converges towards the population mean, so
\begin{equation*}
Y_3 := \frac{d}{nk}\sum_{j=1}^n((k+2-d\|P_{V_j}(x)\|^2)P_{V_j})
\end{equation*}
approximates $Y_2$, and we observe that $Y_3$ lies in the range of $\mathcal{F}_n^*$. In view of tail bound estimates, it will be advantageous to use an additional cut-off similar to the one in \cite{Candes:uq}: keeping in mind that \eqref{eq:equalities for higher moments} yields $\frac{d^2}{k}\mathbb{E}\|P_{V_j}(x)\|^4\rightarrow k+2$, when $d$ tends to infinity, we define the dual certificate by
\begin{equation}\label{dual certificate definition}
Y : = \frac{d}{nk}\sum_{j=1}^n \lambda_j P_{V_j}, \quad\text{where}\quad \lambda_j=\alpha-d\|P_{V_j}(x)\|^21_{E_j},
\end{equation}
$\alpha=\frac{d^2}{k}\mathbb{E}(\|P_{V_j}(x)\|^4 1_{E_j})$, and $E_j = \{ \sqrt{\frac{d}{k}}\|P_{V_j}(x)\|\leq 2\beta_\gamma \}$ for some constant $\beta_\gamma>0$. Obviously, $Y$ is in the range of $\mathcal{F}_n^*$ and, as outlined above, can be considered as an approximation to $Y_1=2I-2P_x$. 

The above definitions will be used throughout the remaining part of this paper. 

\subsubsection{Dual certificate: $Y_T$}\label{sec:5.3}
We shall verify that the dual certificate defined by \eqref{dual certificate definition} satisfies the first condition in \eqref{eq:3}. The following theorem is the analogy to \cite[Lemma 1]{Demanet:2012uq} and \cite[Lemma 2.3]{Candes:2012fk}:
\begin{theorem}\label{th:last one BB}
Let $x\in S^{d-1}$ be fixed. There are constants $c,C>0$ such that, for $n\geq c d$, 
\begin{equation}\label{eq:dual certificate orth}
\|Y_{T}\|_1  \leq \gamma
\end{equation}
with probability at least $1-e^{-Cn}$. 
\end{theorem}
\begin{proof}
First, we suppose that $x=e_1$ and take care of the general case later. We observe that  
$\|Y_T\|_1\leq \sqrt{2}\|Y_T\|_{HS}\leq 2\sqrt{2}\|y\|^2)$, where $y\in\R^d$ is the first column of $Y$ and $\|\cdot\|_{HS}$ denotes the Frobenius norm. We split $P_{V_j}=Q_jQ_j^*$, such that $Q_j\in\R^{d\times k}$ with orthonormal columns. By using 
\begin{equation*}
Z = \sqrt{\frac{d}{k}}(Q_1,\ldots,Q_n)\in\R^{d\times kn},\qquad h = \sqrt{\frac{d}{k}}\begin{pmatrix} \lambda_1 Q^*_1e_1\\ \vdots \\ \lambda_nQ^*_n e_1  \end{pmatrix}\in\R^{kn},
\end{equation*} 
and $h_j=\sqrt{\frac{d}{k}}\lambda_jQ^*_je_1\in\R^k$, for $j=1,\ldots,n$, we see that $\|y\|^2 = \frac{1}{n^2}\|Zh\|^2$. According to Lemma \ref{lemma:2xi}, $\|h_j\|^2=\lambda_j^2\frac{d}{k}\|P_{V_j}(e_1)\|^2$ is sub-exponential, and \cite[Corollary 5.17]{Vershynin:2012fk} implies
\begin{equation}\label{eq:above estimate}
\mathbb{P}\big(\|h\|^2 - \mathbb{E}\|h\|^2\geq n\big) \leq 2e^{-C_1n},
\end{equation}
for some constant $C_1>0$. Since $\alpha\leq k+2$, we observe that there is a constant $C_2>0$ such that $\mathbb{E}\|h\|^2 \leq C_2 n$. Thus, the above estimate \eqref{eq:above estimate} implies that there is a constant $C_3$ such that 
\begin{equation}\label{eq:est a}
\mathbb{P}\big(\|h\|^2\geq n\big) \leq 2e^{-C_3n}.
\end{equation}
For $n>\log(2)/C_3$, the factor $2$ can be put into a constant in the exponential, say $C>0$. 

For $q\in\R^{kn}$ with $\|q\|=1$ and $q=(q_j)_{j=1}^n$, where $q_j\in\R^k$, we obtain
\begin{equation}\label{eq:est b}
\|Zq\|^2 \leq\frac{d}{k}\sum_{j=1}^n \|Q_j q_j\|^2 = \frac{d}{k}\sum_{j=1}^n \|q_j\|^2 =d/k,
\end{equation}
where we have used that the columns of $Q_j$ are orthonormal. By combining \eqref{eq:est a} with \eqref{eq:est b}, we obtain 
\begin{equation*}
\|y\|^2 = \frac{1}{n^2}\|Zh\|^2 = \frac{1}{n^2}\|h\|^2 \|Z\frac{h}{\|h\|}\|^2\leq \frac{d}{kn}
\end{equation*}
with probability at least $1-e^{-Cn}$. Thus, for sufficiently large $c>0$, the condition $n\geq cd$ implies \eqref{eq:dual certificate orth}. 

To conclude the proof, we need to allow general $x\in S^{d-1}$. Note that there exists an orthogonal matrix $U$ such that $x=Ue_1$. We observe that $T_x=U T_{e_1} U^*$ and $P_{UV_j}=U^*P_{V_j} U$. Therefore, the definition $Y_x:=UY_{e_1}U^*$, where $Y_{e_1}$ is the dual certificate w.r.t.~${e_1}$, is in the range of the map $\tilde{\mathcal{F}}_n$ that corresponds to $\{UV_j\}_{j=1}^n$. The latter subspaces are also i.i.d.~according to $\sigma_k$. Since $(Y_x)_{T_x}= UY_{T_{e_1}}U^*$, we also derive $\|(Y_x)_{T_x}\|_1 = \|Y_{T_{e_1}}\|_1$.

\end{proof}

\subsubsection{Dual certificate: $Y_{T^\top}$}\label{sec:5.4}
Let us verify that the dual certificate $Y$ in \eqref{dual certificate definition} satisfies the second condition in \eqref{eq:3}. Indeed, we prove a slightly stronger result:
\begin{theorem}\label{th:second condition new}
Let $x\in S^{d-1}$ be fixed. For all $0<\varepsilon<1/2$, there is $\delta\geq 3/2$ and constants $c,C>0$ such that, for $n\geq c d$, 
\begin{equation}\label{eq:dual certificate orth 2}
\|Y_{T^\perp} - \delta I_{T^\perp}\|_\infty  \leq \varepsilon
\end{equation}
with probability at least $1-e^{-Cn}$. 
\end{theorem}
Note that \eqref{eq:dual certificate orth 2} implies $Y_{T^\perp}\succeq I_{T^\perp}$.  The proof follows the analogous results in \cite[Lemma 2]{Demanet:2012uq} and \cite[Lemma 2.3]{Candes:2012fk}, where $k=1$ is addressed, see our Appendix \ref{app:3.5} for the details.

\subsection{Proof of Theorem \ref{th:finale!}}
After having generalized the intermediate results from $k=1$ to the general case $k\geq 1$, we can assemble these findings as in \cite{Candes:2012fk,Candes:uq,Demanet:2012uq} to prove Theorem \ref{th:finale!}. The details are presented in Appendix \ref{app:4}.

\begin{remark}
Note that our proof for involving semidefinite programming in the phase retrieval problem is guided by the ideas in \cite{Candes:uq,Candes:2012fk,Demanet:2012uq}. Meanwhile the golfing scheme as originally proposed in \cite{Gross:2011fk} has been used for constructing dual certificates in rank-$1$ phase retrieval with semidefinite programming enabling a partial derandomization \cite{Gross:2013fk}, so that $\sigma_{1,d}$ can be replaced with a probability distribution of smaller support. However, signal recovery probability decreases as well. Analogous results also hold for rank-$k$ phase retrieval \cite{Ehler:2014qc}. It is worth mentioning that the golfing scheme was also used to treat phase retrieval with sparse signals \cite{Li:2012uq} and with coded diffraction patterns \cite{Cands:2014bf}. 
\end{remark}

\subsection{Stability} 
In many applications of interest, we may have access to the exact subspaces $\{V_j\}_{j=1}^n$ but the actual measurements are noisy, so that we need to reconstruct the signal from observations of the form
\begin{equation}\label{eq:f}
f_j = \|P_{V_j}(x)\|^2+\omega_j, \quad j=1,\ldots,n,
\end{equation}
where $\omega_j$ is some distortion term. If we replace the feasibility problem of the semi-definite program with the constrained $\ell_1$-minimization 
\begin{equation}\label{eq:l1 min problem}
\arg\min_{X\in\mathscr{H}} \|\mathcal{F}_n(X) - f\|_{\ell_1},\quad\text{subject to}\quad X\succeq 0,
\end{equation}
then we obtain the same stability properties as in \cite{Candes:2012fk}. Indeed, we can straightforwardly follow the lines of the proof in \cite[Theorem 1.3]{Candes:2012fk} for $k=1$ to derive our next statement, which covers $k\geq 1$:
\begin{theorem}\label{th:stab}
There are constants $c_0,c_1,c_2>0$ such that, if $n\geq c_1d$ and $\{V_j\}_{j=1}^n\subset\mathcal{G}_{k,d}$ are chosen independently identically distributed according to $\sigma_k$,  then, for all $x\in\R^d$ and $f$ given by \eqref{eq:f}, the solution $\hat{X}$ to \eqref{eq:l1 min problem} obeys
\begin{equation}\label{eq:stab eq}
\|\hat{X}-xx^*\|_{HS} \leq c_0 \frac{\|\omega\|_{\ell_1}}{n}
\end{equation}
with probability at least $1-e^{-c_2 n}$. 
\end{theorem}
It was also pointed out in \cite{Candes:2012fk} that \eqref{eq:stab eq} implies 
\begin{equation*}
\min \big( \|\hat{x} -x\|,\|\hat{x}+x\|\big) \leq c_0 \min\big(\|x\|,\frac{\|\omega\|_{\ell_1}}{n \|x\|}\big),
\end{equation*}
where $\hat{x}=\sqrt{\alpha} x_0$ and $\alpha$ is the largest eigenvalue of $\hat{X}$ with normalized eigenvector $x_0$. 
Hence, we also have a bound on the deviation to the exact signal when the measurements are noisy and $k>1$.


\section{Numerical experiments}\label{sec:final}

We shall present some numerical experiments illustrating Theorem \ref{th:finale!} and the choice of $k$. Let $x\in S^{d-1}$ and observe that $V\in\mathcal{G}_{k,d}$ is uniformly distributed if and only if $P_V = Z(Z^*Z)^{-1}Z^*$ for some $Z\in\R^{d\times k}$ with independent standard normal entries, cf.~\cite[Theorem 2.2.2]{Chikuse:2003aa}. Thus, we can easily generate pseudo-random orthogonal projectors $\{P_{V_j}\}_{j=1}^n$. Since $\|P_{{V_j}^\perp}(x)\|^2+\|P_{V_j}(x)\|^2=1$, we shall restrict us to $k\leq d/2$. We follow the numerical experiments in \cite{Candes:uq}, where the measurement vector is $ f=(\|P_{V_j}(x)\|^2)_{j=1}^n$. 
As in \cite{Candes:uq}, we use the software package Templates for First-Order Conic Solvers (TFOCS) \cite{Becker:2011fk}.  
If $\hat{X}$ is the solution, then we define $\pm \hat{x}\in S^{d-1}$ as the normalized eigenvector corresponding to the largest eigenvalue of $\hat{X}$. If $x$ is not supposed to lie on the sphere, then we can use the largest eigenvalue to rescale the normalized eigenvector.

\subsection{Examples of signal reconstruction}
We illustrate Theorem \ref{th:finale!} by following a numerical test from \cite{Candes:2011fk}. As in \cite{Candes:2011fk} for $k=1$, the computed approximation is visually indistinguishable from the test signal when $k=10$ and $k=20$, where $d=128$ and $n=6d$, cf.~Fig.~\ref{fig:pure signal}. 
\begin{figure}[h]
\centering
\subfigure[$k=10$]{
\includegraphics[width=.45\textwidth]{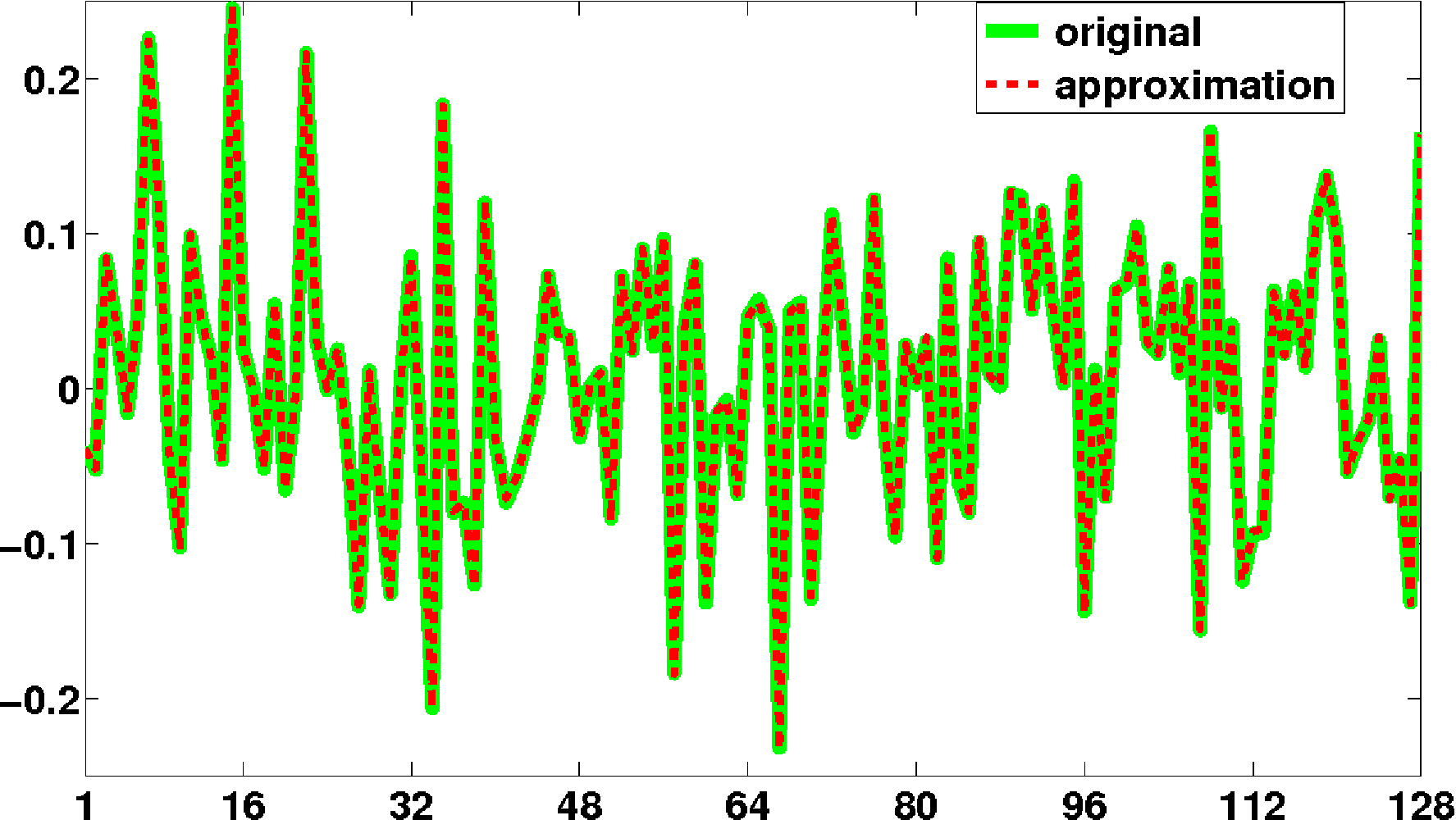}}
\subfigure[$k=20$]{
\includegraphics[width=.45\textwidth]{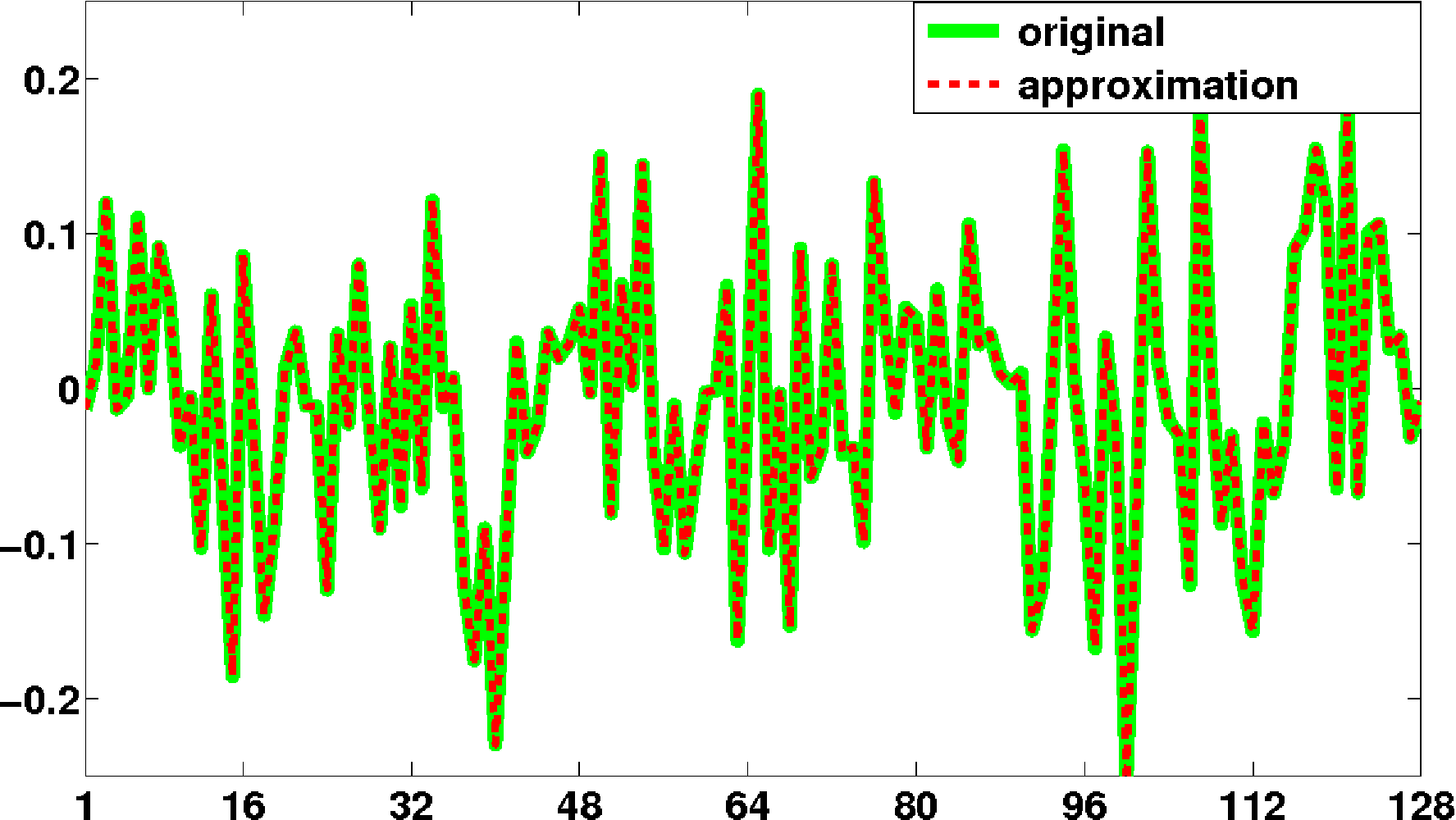}}
\caption{We choose the original signal $x$ uniformly distributed on the sphere $S^{d-1}$. As in \cite{Candes:2011fk}, where $k=1$ was used, the approximation is computed for $d=128$ and $n=6d$. Here, also for $k>1$, we see that original and computed signal are visually indistinguishable.}\label{fig:pure signal}
\end{figure}
\subsection{Optimal choice of $k$}
We investigate on the optimal choice of $k$. Indeed, for $d=6,8,10,12$, we check on the reconstruction rate in dependence of the number of subspaces $n$ when $k$ varies between $1$ and $d/2$. We see in Figure \ref{fig:choice of k} that, for small $n$, the proposed algorithm yields higher recovery rates when $k$ is selected bigger than $1$, and the choice $k=\lceil d/4 \rceil$ appears to be optimal. Here, the recovery rate is computed as the number of reconstructions deviating less than $10^{-2}$ from the original signal divided by the number of repeats ($1000$). 
\begin{figure}[h]
\subfigure[$d=6$]{\includegraphics[width=.4\textwidth]{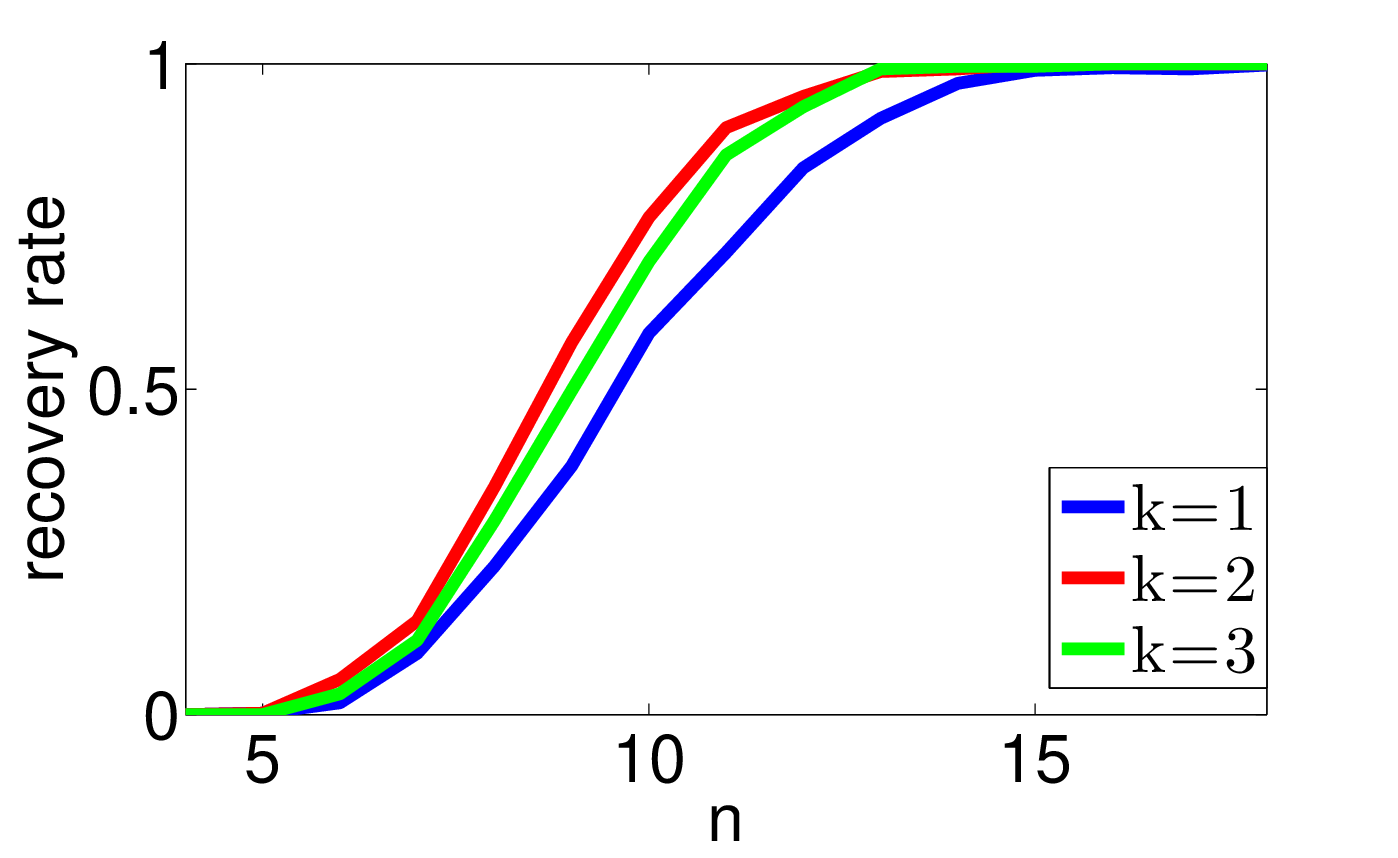}}
\subfigure[$d=8$]{\includegraphics[width=.4\textwidth]{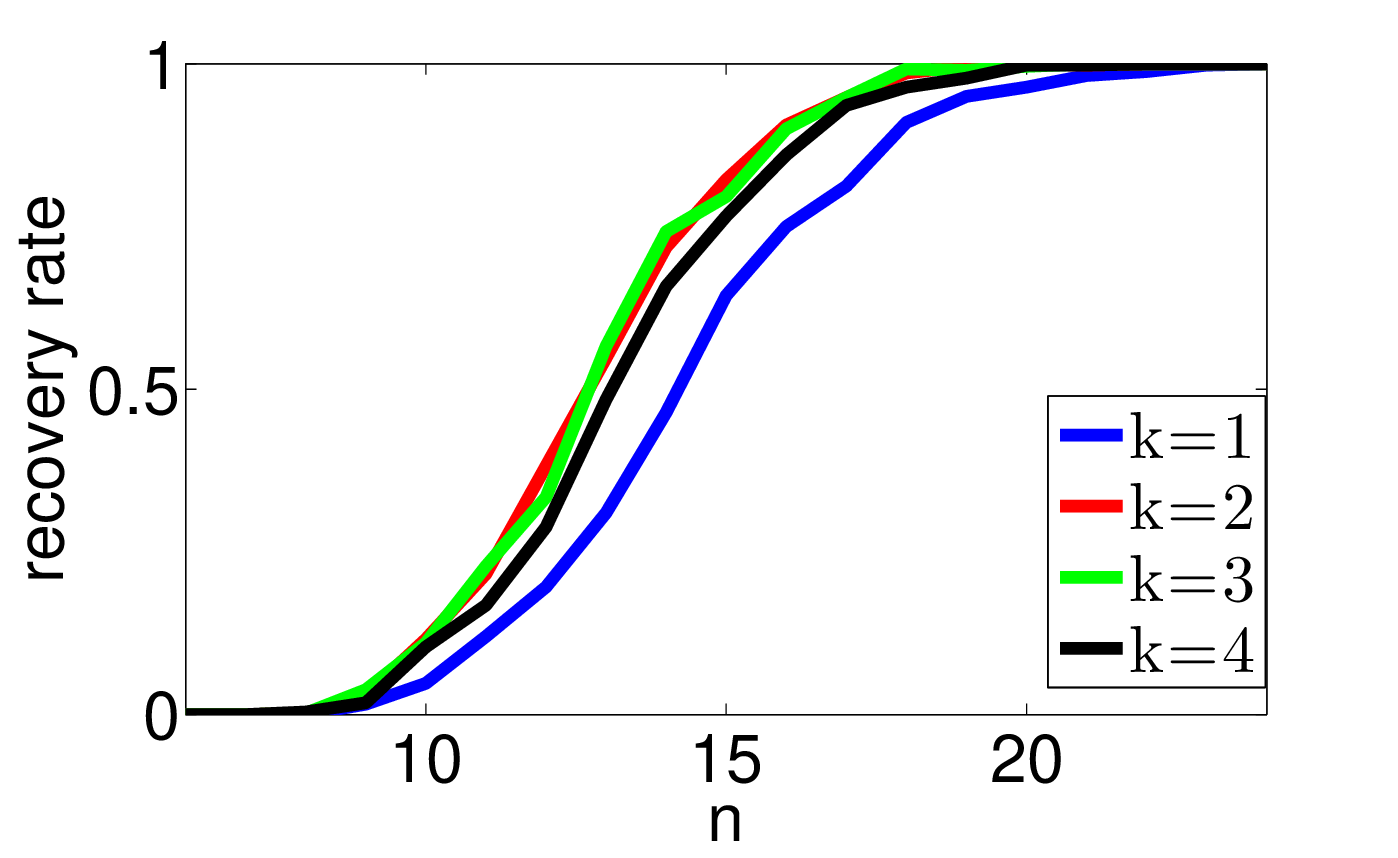}}

\subfigure[$d=10$]{\includegraphics[width=.4\textwidth]{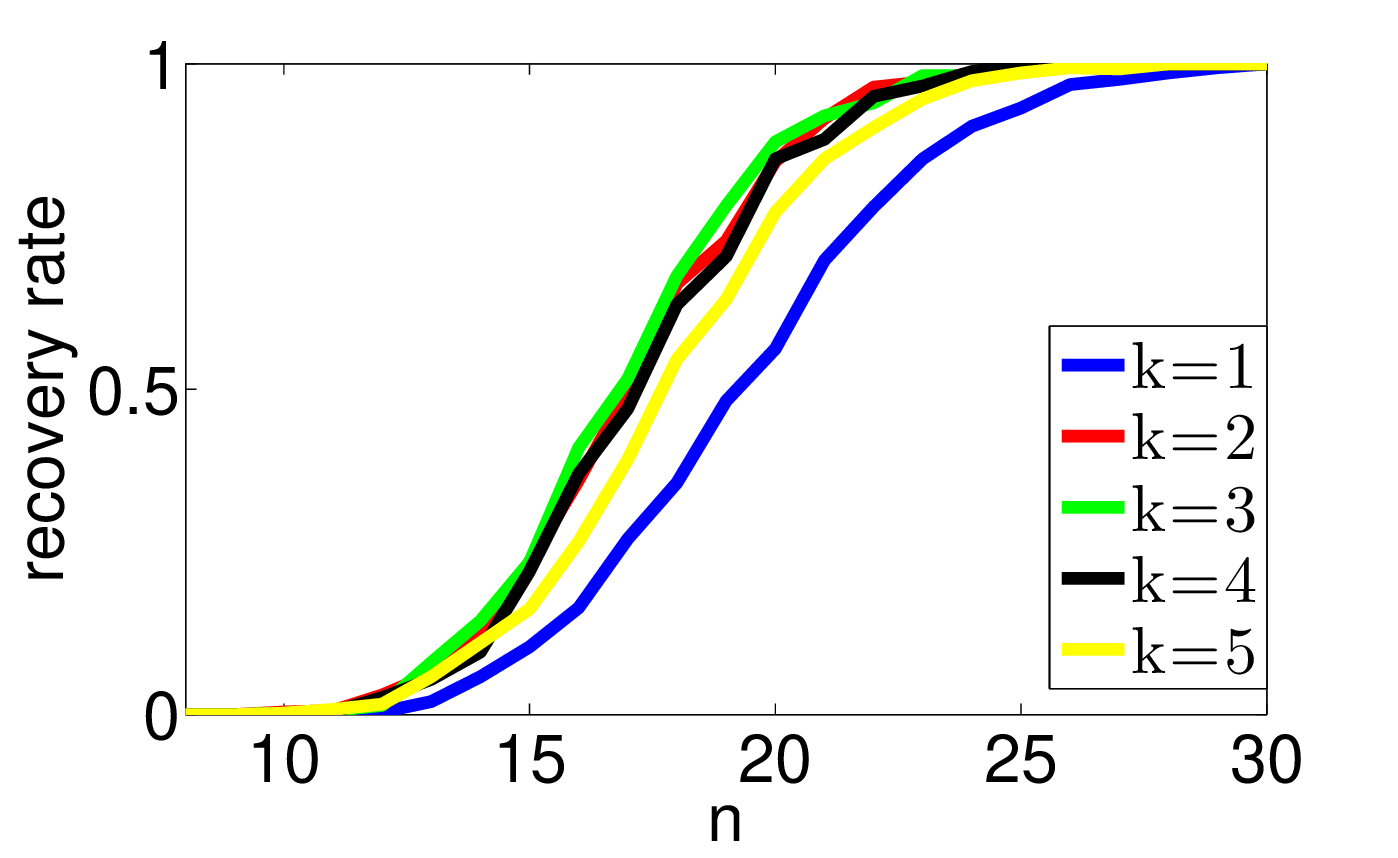}}
\subfigure[$d=12$]{\includegraphics[width=.4\textwidth]{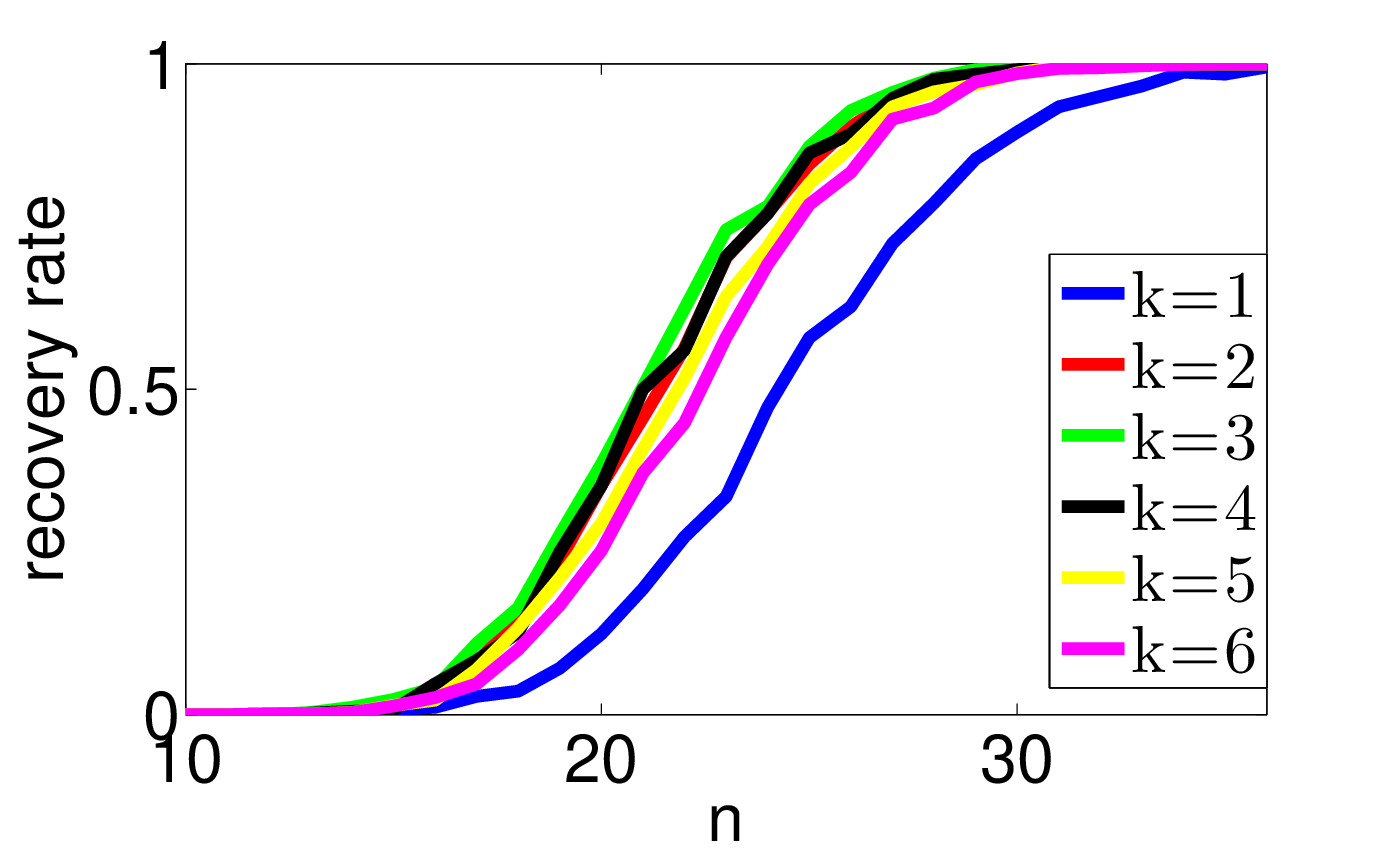}}
\caption{When the subspace number $n$ is small but the subspace dimension $k$ can be selected freely, then $k=1$ is clearly not the optimal choice. It appears that $k=\lceil d/4\rceil$ yields the best results.}\label{fig:choice of k} 
\end{figure}

\section{Brief outline of the complex case}\label{sec:complex}

If we deal with complex signals $x\in\C^d$ and complex $k$-dimensional subspaces $\{V_j\}_{j=1}^n$, then there is again a canonical notion of cubature, cf.~\cite{Roy:2010fk}, and the complex analogue of Proposition \ref{prop:2} holds with adjusted constants $a_1=\frac{(d-1)d(d+1)}{k(d-k)}$ and $a_2=\frac{kd-1}{d-k}$. 

For random subspaces, Theorem \ref{th:finale!} can also be derived in the complex setting. The underlying Theorem \ref{th:candes} holds the same way for complex signals and subspaces, so that we need to verify the respective conditions as in the real case. If the subspaces are chosen i.i.d.~from the Haar measure on the complex Grassmann space, then $\frac{d}{k}\|P_{V_j}(x)\|^2$ is unitarily invariant in $x$ and sub-exponential since  $\mathbb{E}\|P_{V_j}(x)\|^{2p}=\frac{(k)_p}{(d)_p}$, cf.~\cite{Bachoc:2010aa}. Thus, the analogue of Lemma \ref{lemma:2xi} holds. Proposition \ref{lemma:ck} can be extended to the complex case, because the underlying result from \cite[Proposition 7.5]{Eaton:1989fk} has a complex version too.  
The formula \eqref{eq:cubature representation} still holds, only the constants $a$ and $b$ need adjustments, so that  the dual certificate $Y$ can be defined the same way as in \eqref{dual certificate definition}. Thus, we can follow the same proof strategy to cover the complex phase retrieval problem.



\section*{Acknowledgements} 
The authors would like to thank Thomas Bauer for valuable advice on the proof of Proposition \ref{Prop:algebraic system} and Pierre Thibault for discussions on diffraction imaging. The authors are also thankful to Frank Vallentin and Christian Reiher for discussions on the proof of Theorem \ref{th:finale!}. M.~E.~has been funded by the Vienna Science and Technology Fund (WWTF) through project VRG12-009. 

\bibliographystyle{amsplain}
\bibliography{biblio_ehler2}

\appendix 
\section{Proof of Theorem \ref{th:candes}}\label{app:1}
\begin{proof}
For $Z\in\mathscr{H}$ being positive semi-definite and satisfying $\mathcal{F}_n(Z)=f$, we choose $H:=Z-xx^*$ and aim to verify that $H=0$. Since 
\begin{equation*}
0 = \mathcal{F}_n(H)=\mathcal{F}_n(H_T)+\mathcal{F}_n(H_{T^\perp}),
\end{equation*}
and $H_{T^\perp}=Z_{T^\perp}$ is positive semi-definite, the Conditions \eqref{eq:2} and \eqref{eq:1} yield
\begin{equation}\label{eq:H}
A\|H_T\|_\infty \leq \frac{1}{n}\|\mathcal{F}_n(H_T)\|_{\ell_1} = \frac{1}{n}\|\mathcal{F}_n(H_{T^\perp})\|_{\ell_1} \leq B \|H_{T^\perp}\|_1.
\end{equation}
The range of $\mathcal{F}_n^*$ is orthogonal to the nullspace of $\mathcal{F}_n$, so that we derive
\begin{align*}
0 & = \langle H,Y\rangle  = \langle H_{T^\perp},Y_{T^\perp}\rangle+\langle H_T,Y_T\rangle. 
\intertext{The left-hand inequality of \eqref{eq:3} yields}
0 & \geq \langle H_{T^\perp},Y_{T^\perp}\rangle - \gamma \|H_T\|_\infty,
\intertext{and the right-hand inequality of \eqref{eq:3} leads to 
$\|H_{T^\perp}\|_1=\langle H_{T^\perp},I_{T^\perp}\rangle \leq \langle H_{T^\perp},Y_{T^\perp}\rangle$, so that we obtain
}
0 & \geq \|H_{T^\perp}\|_1-\gamma \|H_T\|_\infty \geq (\frac{A}{B}-\gamma)\|H_T\|_\infty,
\end{align*}
where we have used \eqref{eq:H}. Thus, $H_T=0$ must hold and hence also $H_{T^\perp}=0$, so that we have $Z=xx^*$.
\end{proof}

\section{Proof of Theorem \ref{th:singular}}\label{app:2}
The following result extends findings on the smallest and largest singular values $s_{\min}(P)$ and $s_{\max}(P)$ of a random matrix $P$ with independent sub-exponential rows in \cite[Theorem 5.39]{Vershynin:2012fk}. Here, we consider independent blocks but there are dependent rows within each block:
\begin{proposition}\label{th:singular vershynin}
Let $P:=\sqrt{\frac{d}{k}}\big( P_{V_1} ,\ldots, P_{V_n}\big) ^*\in\R^{nd\times d}$, in which $\{V_j\}_{j=1}^n$ are identically and  independently distributed according to $\sigma_k$ on $\mathcal{G}_{k,d}$. Then, for every $t\geq 0$, we have with probability at least $1-2\exp(-ct^2)$
\begin{equation*}
\sqrt{n}-C\sqrt{d}-t\leq s_{\min}(P)\leq s_{\max}(P)\leq \sqrt{n}+C\sqrt{d}+t,
\end{equation*}
where $c,C>0$ are absolute constant.
\end{proposition}



The proof of Proposition \ref{th:singular vershynin} requires two lemmas for preparation:
\begin{lemma}[{\cite[Lemma 5.36]{Vershynin:2012fk}}]\label{lemma:vershynin}
If $B\in\R^{n\times d}$ satisfies 
$ 
\|B^*B - I\|_\infty\leq \max(\delta,\delta^2),
$ 
for some $\delta>0$, then 
\begin{equation}\label{eq:helper}
1-\delta \leq s_{\min}(B) \leq s_{\max}(B)\leq 1+\delta.
\end{equation}
Conversely, if $B$ satisfies \eqref{eq:helper}, then $\|B^*B - I\|_\infty\leq 3\max(\delta,\delta^2)$. 
\end{lemma}
An  $\varepsilon$-net $\mathcal{N}_\varepsilon$ is a finite subset of $ S^{d-1}$ such that to any element $x\in S^{d-1}$, there is an element in $\mathcal{N}_\varepsilon$ at distance less than or equals $\varepsilon$.
\begin{lemma}[{\cite[Lemma 5.4]{Vershynin:2012fk}}]\label{lemma:nets and so on}
Let $A\in\R^{d\times d}$ be symmetric, and let $\mathcal{N}_\varepsilon$ be an $\varepsilon$-net of $S^{d-1}$ for some $\varepsilon\in[0,\frac{1}{2})$. Then
\begin{equation*}
\|A\|_\infty = \sup_{x\in S^{d-1}} |\langle Ax,x\rangle|\leq (1-2\varepsilon)^{-1} \sup_{x\in\mathcal{N}_\varepsilon} |\langle Ax,x\rangle|.
\end{equation*}
\end{lemma}

\begin{proof}[Proof of Proposition \ref{th:singular vershynin}]
To verify Proposition \ref{th:singular vershynin}, we want to apply Lemma \ref{lemma:vershynin} with $B=\frac{1}{\sqrt{n}}P$. We shall explicitly derive the upper estimate on $s_{\max}(P)$. The lower estimate on $s_{\min}(P)$ follows from similar arguments. We must check that
\begin{equation}\label{eq:in proof somewhere}
\|\frac{1}{n}\frac{d}{k} \sum_{j=1}^n P_{V_j}-I\|_\infty\leq \max(\delta,\delta^2)=:\varepsilon,\quad\text{where}\quad \delta=C\sqrt{\frac{d}{n}} + \frac{t}{\sqrt{n}}.
\end{equation}
Let $\mathcal{N}$ be a $\frac{1}{4}$-net, so that an application of Lemma \ref{lemma:nets and so on} yields 
\begin{align*}
\|\frac{1}{n}P^*P-I\|_\infty & =\|\frac{1}{n}\frac{d}{k} \sum_{j=1}^n P_{V_j}-I\|_\infty \\
& \leq 2\max_{x\in\mathcal{N}} |\langle (\frac{1}{n}\frac{d}{k} \sum_{j=1}^n P_{V_j}-I)x,x\rangle |\\
&= 2\max_{x\in\mathcal{N}}|\frac{1}{n} \| Px\|^2 -1|.
\end{align*}
Thus, we must verify with the required probability that
\begin{equation*}
\max_{x\in\mathcal{N}}|\frac{1}{n} \| Px\|^2 -1|\leq \frac{\varepsilon}{2}.
\end{equation*}
To derive this estimate, we define random variables $Z_j=\sqrt{\frac{d}{k}} \|P_{V_j}(x)\|$ so that  
$ 
 \sum_{j=1}^n Z_j^2=\|P(x)\|^2 
$.  
Since $\mathbb{E}(Z_j^2)=1$, we can estimate 
\begin{align*}
\|Z^2_j\|_{\psi_1}&  :=\sup_{p\geq 1}p^{-1}(\mathbb{E}Z_j^{2p})^{1/p} \\
& = \frac{d}{k}\sup_{p\geq 1}p^{-1}(\mathbb{E}\|P_{V_j}(x)\|^{2p})^{1/p} \\
& = \frac{d}{k}\sup_{p\geq 1}p^{-1}(\frac{(k/2)_p}{(d/2)_p})^{1/p}\leq 1,
\end{align*}
where the last inequality is due to Lemma \ref{lemma:2xi}. According to \cite[Remark 5.18]{Vershynin:2012fk}, $\|Z^2_j-1\|_{\psi_1} \leq 2$, and we obtain from the Bernstein type inequality \cite[Corollary 5.17]{Vershynin:2012fk}
\begin{align*}
\mathbb{P}(|\frac{1}{n}\sum_{j=1}^n \frac{d}{k} \|P_{V_j}(x)\|^2-1|\geq \varepsilon/2)  & =\mathbb{P}(|\frac{1}{n}\sum_{j=1}^n Z^2_j-1|\geq \varepsilon/2)  \\
& \leq 2\exp(-cn\min(\frac{\varepsilon^2}{16},\frac{\varepsilon}{4}))\\
& = 2\exp(-n\frac{c}{16}\delta^2)\\
& \leq 2\exp(-\frac{c}{16}(C^2d+t^2)),
\end{align*}
where the last line follows from \eqref{eq:in proof somewhere}. 
Since the net can be chosen such that $|\mathcal{N}|\leq 9^d$, cf.~\cite{Vershynin:2012fk}, we obtain
\begin{equation*}
\mathbb{P}(\max_{x\in\mathcal{N}}|\frac{1}{n}\sum_{j=1}^n \frac{d}{k} \|P_{V_j}(x)\|^2-1|\geq\varepsilon/2) \leq 9^d 2\exp(-\frac{c}{16}(C^2d+t^2))\leq 2\exp(-\frac{c}{16}t^2),
\end{equation*}
where we assume $C\geq 4\sqrt{\ln(9)/c}$. The latter does not cause any trouble because $c$ is a constant independent of $\varepsilon$. This finally yields 
\begin{equation*}
\mathbb{P}(s_{\max}(P)\geq \sqrt{n}+C\sqrt{d}+t)\leq 2\exp(-\frac{c}{16}t^2).
\end{equation*}

The estimates on $s_{\min}(P)$ are derived analogously. 
\end{proof}

We can now prove Theorem \ref{th:singular}:
\begin{proof}[Proof of Theorem \ref{th:singular}]
Since any positive semidefinite matrix $X$ can be written by means of its projectors on eigenspaces, it is sufficient to verify 
\begin{equation*}
1-r\leq \frac{1}{n}\|\mathcal{F}_n(xx^*)\|_{\ell_1} \leq 1+r,\quad\forall x\in S^{d-1},
\end{equation*}
in place of \eqref{eq:l1 norms}. We observe that $\|Px\|^2=\|\mathcal{F}_n( xx^*)\|_{\ell_1}$, so that 
\begin{equation*}
s^2_{\min}(P)\leq \|\mathcal{F}_n(xx^*)\|_{\ell_1}\leq s^2_{\max}(P)
\end{equation*}
holds. First, we take care of the upper bound. According to Proposition \ref{th:singular vershynin}, we have
\begin{equation*}
\frac{1}{n}\|\mathcal{F}_n(P_x)\|_{\ell_1}\leq \frac{1}{n} s^2_{\max}(P)\leq (1+\frac{1}{\sqrt{n}}(C\sqrt{d}+t))^2,
\end{equation*}
with probability at least $1-2e^{-ct^2}$. Choose $\varepsilon>0$ such that $r/4=\varepsilon^2 +\varepsilon$ and observe that $\varepsilon\geq \frac{r}{5}$, so that $n\geq c_1 r^{-2}d$ with $c_1=25C^2$ implies $n\geq \varepsilon^{-2} C^2 d$. For $t=\sqrt{n}\varepsilon$, we obtain that 
\begin{equation*}
\frac{1}{\sqrt{n}}s_{\max}\leq (1+2\varepsilon)
\end{equation*}
holds with probability at least $1-2e^{-cn\varepsilon^2}$. Hence, we have $\frac{1}{n}s^2_{\max}\leq (1+r)$ with the same probability. Since $\varepsilon^2\geq r^2/25$, we can adjust $c_1$ such that $n\geq \frac{25}{cr^2}\ln(2)$ so that $c_2>0$ exists and the required upper estimate holds with probability $1-e^{-c_2r^2 n}$, The lower estimate can be derived in an analogous way. 
\end{proof}

\section{Proof of Theorem \ref{th:ck}} \label{app:3}

\begin{proof}[Proof of Theorem \ref{th:ck}]
It is sufficient  to consider $\|X\|_\infty=1$, so that $X=P_{z_1}-tP_{z_2}$, where $z_1, z_2\in S^{d-1}$ and $z_1\perp z_2$ and $t\in [-1,1]$. We observe
\begin{equation*}
\frac{1}{n}\|\mathcal{F}_n(X)\|_1 = \frac{1}{n}\sum_{j=1}^n \frac{d}{k}|\|P_{V_j}(z_1)\|^2 - t \|P_{V_j}(z_2)\|^2|= \frac{1}{n}\sum_{j=1}^n \xi_j,
\end{equation*}
where $\xi_j=\frac{d}{k}|\|P_{V_j}(z_1)\|^2 - t \|P_{V_j}(z_2)\|^2|. 
$ 
Since $|t|$ is bounded, Lemma \ref{lemma:2xi} implies that $\xi_j$ is sub-exponential. Therefore, the Bernstein inequality as stated in \cite{Vershynin:2012fk} yields
\begin{equation*}
\mathbb{P}(|\frac{1}{n}\|\mathcal{F}_n(X)\|_1-\mathbb{E}\xi|\geq  \varepsilon ) \leq 2\exp(-c n \min(\frac{\varepsilon^2}{4},\frac{\varepsilon}{2})),
\end{equation*}
where $c>0$ is an absolute constant. Proposition \ref{lemma:ck} yields $\mathbb{E}\xi_j\geq u$, and, for $\varepsilon<2$, we derive 
\begin{equation*}
\frac{1}{n}\|\mathcal{F}_n(X)\|_1\geq u-\varepsilon,
\end{equation*}
with probability at least $1-2\exp(-C_1n\varepsilon^2)$, where $C_1=c/4$. The choice $\varepsilon=u r$ establishes the required estimate at least for fixed $X\in T$ with probability at least $1-2\exp(-C_2n r^2)$, where $C_2=C_1u^2$. The remaining part of the proof is the same covering argument as in \cite{Candes:uq}, so we omit this.
\end{proof}

\section{Proof of Theorem \ref{th:second condition new}}\label{app:3.5}
\begin{proof}[Proof of Theorem \ref{th:second condition new}]
As in the proof of Theorem \ref{th:last one BB}, we first consider $x={e_1}$. Let us split $Y=Y^{(0)}-Y^{(1)}$ into
\begin{equation*}
Y^{(0)} =\frac{1}{n}\sum_{j=1}^n \alpha\frac{d}{k}P_{V_j},  \qquad Y^{(1)} = \frac{1}{n}\sum_{j=1}^n d\|P_{V_j}(e)\|^21_{E_j}\frac{d}{k}P_{V_j}.
\end{equation*}
First, we shall estimate $\|Y^{(0)}_{T^\perp}-\alpha I_{T^\perp}\|_\infty$, later also $\|Y^{(1)}_{T^\perp}-b_0I_{T^\perp}\|_\infty$ for some special number $b_0$. We observe that $\mathbb{E}Y^{(0)} = \alpha I$. By using $P:=\sqrt{\frac{d}{k}}\big( P_{V_1} ,\ldots, P_{V_n}\big) ^*$ as in Proposition \ref{th:singular vershynin} and squaring the estimates there, we see that 
\begin{equation*}
(\sqrt{n}-C_1\sqrt{d}-t)^2\leq s^2_{\min}(P)\leq s^2_{\max}(P)\leq (\sqrt{n}+C_1\sqrt{d}+t)^2
\end{equation*}
with probability at least $1-2e^{-c_1t^2}$. 
Since $\frac{\alpha}{n}P^*P=Y^{(0)} $, the latter implies at least for sufficiently small $t/\sqrt{n}$:
\begin{equation*}
\|Y^{(0)}-\alpha I\|_\infty \leq \alpha(C_1^2d+t^2+2\sqrt{nd}+2\sqrt{n}t+2C_1t\sqrt{d})
\end{equation*}
with the same probability.  
For all $\varepsilon_1>0$, there is $c_2$ sufficiently large and $\varepsilon_2>0$ sufficiently small such that $t=\varepsilon_2\sqrt{n}$ yields
\begin{equation*}
\|Y^{(0)}-\alpha I\|_\infty \leq  \alpha\varepsilon_1,
\end{equation*}
for all $n\geq c_2 d$ with probability $1-e^{-c_3n}$. In particular, we have 
\begin{equation}\label{est a}
\|Y^{(0)}_{T^\perp}-\alpha I_{T^\perp}\|_\infty \leq  \alpha\varepsilon_1
\end{equation}
with the same probability. 

Let us now take care of $Y^{(1)}_{T^\perp}$. Due to the unitary invariance of $\sigma_k$, \eqref{eq:cubature representation} for $X=P_{e_1}$ yields 
\begin{equation*}
\mathbb{E}(d\|P_{V_j}({e_1})\|^2 1_{E_j} \frac{d}{k}P_{V_j}) =a_0P_{e_1}+b_0I,
\end{equation*}
for some constants $a_0, b_0>0$ that depend on $\beta_\gamma$. 
Therefore, we have 
$ 
\mathbb{E}Y^{(1)}_{T^\perp} =b_0 I.
$ 
The random matrix 
\begin{equation*}
X_j=\frac{d^2}{k}\|P_{V_j}({e_1})\|^2 1_{E_j}(P_{V_j})_{T^\perp}-b_0 I_{T^\perp}
\end{equation*}
is bounded, say by $K$. We find a constant $C_2>0$ such that $\|\mathbb{E}X_j^*X_j\|_\infty\leq C_2$ implying $\|\sum_{j=1}^n \mathbb{E}X_j^*X_j\|_\infty\leq nC_1$. According to \cite[Theorem 5.29]{Vershynin:2012fk}, we have, for all $t>0$,
\begin{equation*}
\mathbb{P}\big( \|\frac{1}{n}\sum_{j=1}^n X_j\|_\infty\geq \frac{t}{n} \big)\leq 2 d e^{\frac{-t^2/2}{nC_2+Kt/3}}.
\end{equation*}
By choosing $\varepsilon_2>0$ and $t=\varepsilon_3 n$, we derive 
\begin{equation*}
\mathbb{P}\big( \|\frac{1}{n}\sum_{j=1}^n X_j\|_\infty\geq \varepsilon_2 \big)\leq 2 d e^{-c_4n}\leq e^{-c_5n},
\end{equation*}
for all $n\geq c_6\ln(d)$. Thus, we obtain
\begin{equation}\label{est b}
\|Y^{(1)}_{T^\perp}-b_0I_{T^\perp}\|_\infty \leq \varepsilon_2 ,
\end{equation}
with probability $1-e^{-c_5n}$, for all $n\geq c_6\ln(d)$.

Combining \eqref{est a} and \eqref{est b} implies
\begin{equation*}
\|Y_{T^\perp} -(\alpha-b_0)I_{T^\perp}\|\leq \alpha\varepsilon_1 + \varepsilon_2
\end{equation*}
with probability at least $1-e^{-Cn}$, for all $n\geq c d$. We can now choose $\varepsilon_1,\varepsilon_2$ sufficiently small, such that $\alpha\varepsilon_1 + \varepsilon_2\leq \varepsilon$. The term $\alpha$ is bounded by $k+2$. According to Vershynin's lecture note on nonasymptotic random matrix theory (Lemma 9 in Lecture 4 on dimension reduction), we have, for all $\beta_\gamma\geq 1/2$ that $\mathbb{P}\big(E^c_j\big)\leq 2e^{k/2}e^{-k\beta_\gamma}$. Since $\mathbb{E}(\frac{d^4}{k^2}\|P_{V_j}({e_1})\|^8)$ is bounded independently of $d$, see \eqref{eq:equalities for higher moments}, the term $k+2-\alpha=\mathbb{E}(\frac{d^2}{k}\|P_{V_j}({e_1})\|^4 1_{E^c_j})$ can be made arbitrarily small by choosing $\beta_\gamma$ sufficiently large. Thus, we can derive $\alpha\geq k+5/3$. Similar arguments yield that $b_0$ gets closer to $b$ when we increase $\beta_\gamma$. With $b\leq k$ we can assume that $b_0\leq k+1/6$, so that $\delta=\alpha-b_0\geq 3/2$. 

We still need to address general vectors $x\in S^{d-1}$. With the notation and arguments at the end of the proof of Theorem \ref{th:last one BB}, we observe that $\|(Y_x)_{T_x^\perp}-\delta I_{T_x^\perp}\|_\infty = \|(Y_{e_1})_{T_{e_1}^\perp}-\delta I_{T_{e_1}^\perp}\|_\infty$, which concludes the proof.  
\end{proof}

\section{Proof of Theorem \ref{th:finale!}}\label{app:4}
We can now assemble all of our findings to verify that the conditions in Theorem \ref{th:candes} hold with the required probability:
\begin{proof}[Proof of Theorem \ref{th:finale!}]
We first fix $x\in S^{d-1}$. Then we choose $r\in(0,1)$ and $\gamma<u\frac{1-r}{1+r}$, where $u\in(0,1)$ as in Proposition \ref{lemma:ck}. Let $c_i$ and $C_i$, $i=1,\ldots, 4$, be suitable positive constants. 
Theorem \ref{th:singular} yields that Condition \eqref{eq:1} holds with probability of failure at most $e^{-C_1 n}$, for all $n\geq c_1 d$. Theorem \ref{th:ck} implies that Condition \eqref{eq:2} holds with probability of failure at most $e^{-C_2 n}$, for all $n\geq c_2d$. According to Theorem \ref{th:last one BB}, the first condition in \eqref{eq:3} holds with probability of failure at most $e^{-C_3 n}$, for all $n\geq c_3d$. Theorem \ref{th:second condition new} yields that the second condition in \eqref{eq:3} is satisfied with probability of failure at most $e^{-C_4 n}$, for all $n\geq c_4d$. 

Finally, there are constants $c,C>0$ such that, for all $n\geq cd$, we can estimate $\sum_{i=1}^4 e^{-C_i n} \leq e^{-Cn}$, so that all conditions in Theorem \ref{th:candes} are satisfied with probability at least $1-e^{-Cn}$. In order to turn the latter into a uniform estimate in $x$, we take an $\epsilon$-net $\mathcal{N}_{\epsilon}$ on the sphere of cardinality less or equals $(1+\frac{2}{\epsilon})^d$, cf.~\cite[Lemma 5.2]{Vershynin:2012fk}. Since $ (1+\frac{2}{\epsilon})^de^{-Cn} \leq e^{-\tilde{C}n}$, for all $n\geq \tilde{c}d$ when $\tilde{C}$ is sufficiently small and $\tilde{c}$ sufficiently large, we have a uniform estimate for the net $\mathcal{N}_{\epsilon}$. Now, to any arbitrary $x\in S^{d-1}$, we find $x_0\in \mathcal{N}_{\epsilon}$ with $\|x-x_0\|\leq \epsilon$. By following the lines in \cite[Proof of Theorem 1.2]{Candes:2012fk}, one can derive that the certificate for $x_0$ also works for $x$, so that we can conclude the proof of Theorem \ref{th:finale!}.
\end{proof}

\end{document}